\documentclass{article}

\usepackage{amsmath,amsfonts,amssymb,amsthm}
\usepackage{graphics,epsfig,graphicx,graphics}
\usepackage{color}

\usepackage{mathrsfs}
\usepackage{latexsym}
\usepackage{xcolor}

\usepackage{tikz}
\usetikzlibrary{patterns}
\usepackage{hyperref}

\usepackage{vmargin}
\setmarginsrb{2cm}{1cm}{2cm}{2cm}{1cm}{1cm}{1cm}{1cm}

\def\Om{\Omega}
\def\E{\mathcal{E}}
\def\R{\mathbb{R}}
\def\ep{\varepsilon}

\numberwithin{equation}{section}
\numberwithin{figure}{section}

\newcommand{\FF}{{\text{$\sf{E}$}}}
\newcommand{\nn}{{\text{$\sf{n}$}}}
\newcommand{\epp}{\epsilon}

\newcommand{\cn}{\operatorname{cn}}
\newcommand{\sn}{\operatorname{sn}}

\newcommand{\cotan}{\operatorname{cotan}}
\newcommand{\diam}{\operatorname{diam}}

\renewcommand{\leq}{\leqslant}
\renewcommand{\le}{\leqslant}
\renewcommand{\geq}{\geqslant}
\renewcommand{\ge}{\geqslant}

\newtheorem{Theorem}{Theorem}[section]

\newtheorem{Corollary}[Theorem]{Corollary}
\newtheorem{Proposition}[Theorem]{Proposition}
\newtheorem{Lemma}[Theorem]{Lemma}
\newtheorem{Remark}[Theorem]{Remark}

\title{Elastic energy of a convex body}
\author{Chiara Bianchini, Antoine Henrot, Tak\'eo Takahashi}
\date{\today}

\usepackage{amsmath}

\begin{document}
\maketitle

\begin{abstract}
In this paper a Blaschke-Santal\'o diagram involving the area, the perimeter and the elastic energy of planar convex bodies is considered.
More precisely we give a description of set 
$$\mathcal{E}:=\left\{(x,y)\in \R^2, x=\frac{4\pi A(\Omega)}{P(\Omega)^2},
y=\frac{E(\Omega)P(\Omega)}{2\pi^2},\,\Omega\mbox{ convex} \right\},
$$ where $A$ is the area, $P$ is the perimeter and $E$ is the elastic energy, that is a Willmore type energy in the plane.
In order to do this, we investigate the following shape optimization problem:
$$
\min_{\Omega\in\mathcal{C}}\{E(\Omega)+\mu A(\Omega)\},
$$ where $\mathcal{C}$ is the class of convex bodies with fixed perimeter and $\mu\ge 0$ is a parameter.
Existence, regularity and geometric properties of solutions to this minimum problem are shown.
\end{abstract}



\tableofcontents

\section{Introduction}
For a regular planar convex body $\Omega$, that is a planar convex compact set, we introduce the three geometric
quantities $A(\Omega)$, $P(\Omega)$, $E(\Omega)$ where $A(\Omega)$ is the area, $P(\Omega)$ is the perimeter and
$E(\Omega)$ is the elastic energy defined by
$$
E(\Omega)=\displaystyle \frac 12 \int_{\partial \Omega} k^2(s) \ ds
$$
where $k$ is the curvature and $s$ is the arc length.
The elastic energy of a curve seems to have been introduced by L. Euler in 1744 who
studied the {\it elasticae}. These curves are critical points of the elastic energy
which satisfy some boundary conditions. This question has been widely studied and has many
applications in geometry, in kinematics (the ball-plate problem), in numerical analysis
(non-linear splines), in computer vision (reconstruction of occluded edges) etc.
For a good overview and historical presentation, we refer e.g. to \cite{Sat}.

The aim of this paper is to study the links between $E(\Omega)$, $A(\Omega)$ and $P(\Omega)$.
This can be done by investigating the set of points in $\R^3$ corresponding to the triplet
$(A(\Omega),E(\Omega),P(\Omega))$ or a planar scale invariant version as
$(A(\Omega)/P(\Omega)^2, E(\Omega)P(\Omega))$. The first one who studied the diagram of these points is
probably W. Blaschke in \cite{bla} where the three quantities in consideration were
the volume, the surface area and the integral of the mean curvature of a three-dimensional
convex body. Later on, L. Santal\'o in \cite{san} proposed a systematic study of this kind
of diagrams for planar convex body and geometric quantities like the area, the perimeter, the
diameter, the minimal width, the inradius and the circumradius. From that time, this kind
of diagram is often called {\it Blaschke-Santal\'o diagram}.

Our aim is to study the following Blaschke-Santal\'o diagram involving area, elastic
energy and perimeter:
\begin{equation}\label{setE}
\mathcal{E}:=\left\{(x,y)\in \R^2, x=\frac{4\pi A(\Omega)}{P(\Omega)^2},
y=\frac{E(\Omega)P(\Omega)}{2\pi^2},\,\Omega\mbox{ convex} \right\}.
\end{equation}
This will be done in Section \ref{secE}.

For this analysis, we recall an important geometric inequality due to Gage in \cite{Gage}: 
\begin{Theorem}[Gage]\label{main-th1}
 For any planar convex body of class $C^1$ and piecewise $C^2$, the following inequality holds
 \begin{equation}
   \label{main_ineq}
   \frac{E(\Omega) A(\Omega)}{P(\Omega)} \geq \frac{\pi}{2}
 \end{equation}
with equality if and only if $\Omega$ is a disk.
\end{Theorem}
In other words, the disk minimizes the product $E(\Omega) A(\Omega)$
among convex bodies with given perimeter. More general inequalities involving different functions of the curvature,
area and perimeter have been proved in \cite{Gre-Osh}.
Notice that Gage's result implies that the
points $(x,y)$ in $\mathcal{E}$ satisfy the inequality $xy\geq 1$. 
In order to describe the diagram  $\mathcal{E}$ we need additional relations which lead us
 to consider the following minimization problem:
\begin{equation}
  \label{pb2}
  \min_{\Omega \in \mathcal{C}} \left(E(\Omega)+\mu A(\Omega)\right),
\end{equation}
where $\mu\geq 0$ is a parameter and  $\mathcal{C}$ is the class of regular planar convex bodies $\Omega$ such that
$P(\Omega)=P_0$.
We stress that there is a competition between the two terms since the disk minimizes $E(\Omega)$, see below (\ref{minper})
while it maximizes $A(\Omega)$ by the isoperimetric inequality.
Thus we can expect that the penalization parameter $\mu$ plays an important role
and that the solution is close to the disk when $\mu$ is small while it is close to the segment when $\mu$ is large. 
More precisely, we will present several results
in Section \ref{secdisk}. Our objective is to
describe the  boundary of the set $\mathcal{E}$ defined in (\ref{setE}) by solving this minimization problem.

Before tackling this minimization problem, let us make some observations
about the minimization of the elastic energy $\min\{E(\Omega)\}$.
Without any constraint this problem has no solution. Indeed if we consider
a disk $D_r$ of radius $r$, the curvature $k$ is constant equal to $1/r$ so that
$$
\int_{\partial D_r} k^2 \ ds = \frac{2\pi}{r} \to 0 \quad
\text{as} \ r\to \infty.
$$
Now, if we add a constraint of perimeter $P(\Omega)=P_0$, and if we consider that $\Omega$
is a bounded simply connected domain then by using the Cauchy-Schwarz
inequality we deduce
\begin{equation}\label{minper}
2\pi= \int_{\partial \Omega} k \ ds
\leq \left(\int_{\partial \Omega} k^2 \ ds \right)^{1/2} \left(P(\Omega)\right)^{1/2}
\end{equation}
with equality only in the case of a disk. Thus the disk solves
\begin{equation}\label{pb3}
\min \{E(\Omega), P(\Omega)=P_0\}
\end{equation}
among simply connected domains.
Let us remark that the equality constraint $P(\Omega)=P_0$ can be replaced by
an inequality $P(\Omega)\leq P_0$ since $E(t\Omega)=E(\Omega)/t$.
Moreover, if $\Omega$ is not simply connected, the result still holds true since removing
extra parts of the boundary makes the perimeter and the elastic energy lower.

Now if we consider the minimization of $E(\Omega)$ with a constraint on the 
area, there is no minimum. Indeed we can
take the annulus of radii $r$ and $r+\delta_r$ so that the area
constraint is satisfied, then
$$
\int_{\partial \Omega} k^2 \ ds  = 2\pi
\left(\frac{1}{r}+\frac{1}{r+\delta_r}\right) \to 0\quad
\text{as} \ r\to \infty.
$$
On the other hand, this problem  has a solution in the class of convex bodies. This is an easy consequence of Gage's inequality 
together with the isoperimetric inequality:
the disk is the unique minimizer of
the elastic energy under a constraint of area among convex bodies.

To our knowledge, the question to look for a minimizer for the elastic energy among
simply connected sets of given area remains open. Let us also mention
some related works. In \cite{Sat2}, Yu. L. Sachkov
 studies the ``closed elasticae'', that is the closed curves which are
stationary points of the elastic energy. He obtains only two possible curves:
the disk or the ``eight elasticae'' which is a local minimum. His method relies
on optimal control theory and Pontryagin Maximum Principle. The problem of minimizing $E(\Omega)$ among sets with given
$P(\Omega)$ and $A(\Omega)$ has also been studied. Indeed, this problem is related to the modelling
of vesicles which attracts much attention recently. For a study of critical
points of the functional and some numerical results, we refer to \cite{bir}.

\smallskip
Let us mention that this kind of problem has a natural extension in higher
dimension, the elastic energy being replaced by the Willmore functional.
This one being scale invariant, the nature of the problem is different.
For a physical point of view, it is a much more realistic model for vesicles.
For example, the problem of minimizing the Willmore functional (or the Helfrich functional
which is very similar) among three-dimensional sets, with constraints on the
volume and the surface area, is a widely studied problem.

\medskip
The plan of the paper is as follows: in Sections \ref{secex}, \ref{secgeom}, \ref{secdisk}
we study the minimization problem \eqref{pb2}. First existence and $C^2$ regularity of a minimizer
is proved, then some geometric properties are given: symmetry, possibilities of segments
on the boundary and the case of the disk is investigated (for what values of $\mu$ is
it solution or not). In the two last sections, the Blaschke-Santal\'o diagram of
the set $\mathcal{E}$ is investigated, first from a theoretical point of view in Section
\ref{secE} and then from a numerical point of view in Section \ref{secnumeric}.

\subsection{Notations}
For points $M,Q$ in the plane, we indicate by $\overrightarrow{QM}$ the planar vector joining these two points and we denote by $\|\cdot\|$ the Euclidean norm in $R^{N}$.

For an integer $p\geq 1$ and a real number $q\geq 1$, the Sobolev space $W^{p,q}(a,b)$ is the subset of functions $f$ in $L^q(a,b)$ such that the function $f$ and its weak derivatives up to the $p-th$  order belong to the space $L^q(a,b)$.
By $W_0^{p,q}(a,b)$ we indicate the closure in $W^{p,q}(a,b)$ of the infinitely differentiable functions compactly supported in $(a,b)$.
We indicate by $\langle\cdot,\cdot\rangle_{L^2(a,b)}$ the scalar products in the Hilbert space $L^2(a,b)$ and by $\|\cdot\|_{L^2(a,b)}$ its operator norm.

\section{Existence and regularity}\label{secex}
\subsection{Existence}
We recall that $A$, $P$ and $E$ can be expressed in different ways depending on which
parametrization is considered.
Indeed choosing the arc length $s$ parametrization, the area and the
elastic energy can be written in terms of the angle $\theta(s)$ (angle between
the tangent and the horizontal axis) in the following way:
\begin{equation}\label{paramtheta}
E(\Omega)=\frac 12 \int_0^P {\theta'}^2(s)\,ds \qquad A(\Omega)=\int\int_T \cos(\theta(u))\sin(\theta(s))\,
du\,ds
\end{equation}
where $T$ is the triangle $T=\{(u,s)\in \mathbb{R}^2 \ ; \ 0\leq u\leq s\leq P\}$. In this case, we recall that
\begin{equation}\label{bht5.1}
\partial \Omega=\Big\{ (x(s),y(s)), \quad s\in [0,P] \Big\},
\end{equation}
and
\begin{equation}\label{eq11}
x'(s)=\cos\theta(s), \ y'(s)=\sin \theta(s).
\end{equation}
The convexity of the set $\Omega$ is expressed by the fact that the function
$s\mapsto \theta(s)$ is non-decreasing. Notice that expression \eqref{paramtheta} for the
elastic energy leads us to impose the following regularity condition on (the boundary of) the convex set
$\Omega$; that is  the function $\theta(s)$ has to belong to the Sobolev space $W^{1,2}(0,P)$ 
Let us remark that if $\theta$ is
given, we recover the boundary of the convex set by integrating $\cos\theta$ and $\sin\theta$.

Hence let us consider the following class of convex sets:
\begin{equation}\label{bht5.0}
\mathcal{C}:=\left\{ \Omega \subset \mathbb{R}^2 \ \text{bounded and open set such that \eqref{bht5.1} and \eqref{eq11} hold and}\ \theta \in \mathcal{M}\right\}, 
\end{equation}
where
\begin{equation}\label{bht5.2}
\mathcal{M}:=\Bigg\{\theta \in W^{1,2}(0,P) \ ; \ \theta(0)+2\pi=\theta(P),\; \theta'\geq 0 \; a.e., \quad
\int_0^{P} \cos(\theta(s)) \ ds=\int_0^{P} \sin(\theta(s)) \ ds=0\Bigg\}.
\end{equation}

\medskip
On the other hand, choosing the parametrization of the convex set by its support function $h(t)$
$(t\in [0,2\pi])$ and its
radius of curvature $\phi=h''+h \geq 0$, with $\phi=1/k$, we have
\begin{equation}\label{paramphi}
P(\Omega)=\int_0^{2\pi} h(t)dt=\int_0^{2\pi} \phi(t)dt,\quad
A(\Omega)=\frac 12 \int_0^{2\pi} h(t)\phi(t)\,dt, \quad
E(\Omega)=\frac 12 \int_0^{2\pi} \frac{1}{\phi(t)}\,dt ;
\end{equation}
this last expression being valid as soon as $\Omega$ is $C^2_+$ meaning that the
radius of curvature is a positive continuous function and the fact that $ds=\phi(t)\,dt$ (where $s$ is the curvilinear abscissa). 

\begin{Remark}
We underline that if the domain $\Omega$ is not strictly convex or not of class $C^2$, it is well known that its convexity
is just expressed by the fact that $h''+h$ is a non-negative measure.
This is actually a consequence of the Minkowski existence Theorem, see \cite[Section 7.1]{schneider}. 

In this general case the expression of $E(\Omega)$ in \eqref{paramphi} is no longer valid. For more results and properties
of the support function, we refer again to \cite{schneider}.
\end{Remark}

Let us remark also that, for a regular convex body, the radius of curvature $\phi$
is positive (because $\phi=0$ would mean that the curvature $k$ is infinite).
In what follows, we use the operator $G$ defined by $G\phi=h$ where
$h$ is the solution of
\begin{equation}
  \label{eq:sturm}
  h''+h=\phi \ \text{in} \ (0,2\pi), \, h \  2\pi -
  \text{periodic}, \; \int_0^{2\pi} h(t)  \cos(t)\,dt =\int_0^{2\pi} h(t)  \sin(t)\,dt =0.
\end{equation}
Hence the area of $\Omega$ can be rewritten as
\begin{equation}
  A(\Omega)=\frac 12 \int_0^{2\pi} G\phi(t)\phi(t)\,dt.\label{paramphi2}
\end{equation}

\medskip
Without loss of generality and to simplify the presentation, we assume
from now on that the perimeter constraint is
\begin{equation}\label{per2pi}
P(\Omega)=2\pi.
\end{equation}

\medskip
Using the parametrization in $\theta$, Problem \eqref{pb2} can be written as
\begin{equation}\label{pb2-theta}
\inf_{\theta \in \mathcal{M}} j_\mu(\theta),
\end{equation}
where $\mathcal{M}$ is defined by \eqref{bht5.2} with $P=2\pi$
and
$$
j_\mu(\theta):=\frac 12 \int_0^{2\pi} {\theta'}^2(s)\,ds +\mu \int\int_T \cos(\theta(u))\sin(\theta(s))\,du\,ds
$$
with
$$
T=\{(u,s)\in \mathbb{R}^2 \ ; \ 0\leq u\leq s\leq 2\pi\}.
$$
Classical arguments allow to prove the existence of a minimum to problem \eqref{pb2}:
\begin{Theorem}\label{teoexi}
For all $\mu\geq 0$, there exists $\Omega^*\in \mathcal{C}$ which minimizes $J_\mu(\Omega)=
E(\Omega)+\mu A(\Omega)$.
\end{Theorem}
\begin{proof}
Let $\theta_n\in \mathcal{M}$ corresponding to a minimizing sequence of domains $\Omega_n$. Necessarily $\theta_n$ is bounded
in $W^{1,2}(0,2\pi)$, therefore we can extract a subsequence which converges weakly to some $\theta^*$
in $W^{1,2}(0,2\pi)$ and uniformly in $C^0([0,2\pi])$ (the embedding $W^{1,2}(0,2\pi)\hookrightarrow
C^0([0,2\pi])$ being compact). Thus $\theta^*$ is non-decreasing, $\theta^*(0)+2\pi=\theta^*(2\pi)$ and
$$
\int_0^{2\pi} \cos(\theta^*(s)) \ ds=\int_0^{2\pi} \sin(\theta^*(s)) \ ds=0.
$$
Moreover,  $J_\mu(\theta^*)\leq \lim\inf J_\mu(\theta_n)$ which proves the result.
\end{proof}

\subsection{Optimality conditions and regularity}
We want to characterize the optimum of the problem
 \begin{equation}
   \label{jmin}
   \min_{\Omega \in \mathcal{C}} E(\Omega) +\mu A(\Omega),
 \end{equation}
 where $\mathcal{C}$ is defined by \eqref{bht5.0} and where \eqref{per2pi} holds.
 
We first write optimality condition for \eqref{jmin} by considering the parametrization $(x(s),y(s))$ of $\partial \Omega$ such that
\begin{equation}\label{xy}
x'(s)=\cos \theta(s), \quad y'(s)=\sin \theta(s), \quad (s\in [0,2\pi]).
\end{equation}
Without loss of generality, we may assume
\begin{equation}\label{ci}
x(0)=0, \quad y(0)=0.
\end{equation}

The following result gives the main properties of a function $\theta$ associated to an optimal
domain $\Omega$.
\begin{Theorem}\label{Ththeta}
 Assume $\theta$ is associated to an optimal domain $\Omega$ solution of \eqref{jmin}.
 Then $\theta\in W^{2,\infty}(0,2\pi)$ and there exist Lagrange multipliers $\lambda_1$, $\lambda_2$
 and a constant $C$ such that, for all $s\in [0,2\pi]$
\begin{equation}\label{first-opt}
\theta'(s) = \frac{\mu}{2} \left(\frac{\lambda_1^2+\lambda_2^2}{\mu^2}+\frac{2C}{\mu}-
\left[x(s)-\frac{\lambda_2}{\mu}\right]^2-\left[y(s)+\frac{\lambda_1}{\mu}\right]^2 \right)^-
\end{equation}
where $(\cdot)^-$ denotes the negative part of a real number.
\end{Theorem}

We postpone the proof of Theorem \ref{Ththeta} at the end of the section.

\begin{Remark}
On the strictly convex parts of the boundary of $\Omega$, \eqref{first-opt} writes
\begin{equation}\label{opcs}
\theta'(s) = \frac{\mu}{2} \left(
\left[x(s)-\frac{\lambda_2}{\mu}\right]^2+\left[y(s)+\frac{\lambda_1}{\mu}\right]^2 -
\frac{\lambda_1^2+\lambda_2^2}{\mu^2}-\frac{2C}{\mu}\right)
\end{equation}
and  by a classical bootstrap argument, this shows that $\theta$ is indeed $C^\infty$. 
In the non-strictly convex case  there may be a lack of regularity due to the connection points between segments and strictly convex parts.

For similar regularity results for shape optimization problems
with convexity constraints, in a more general context, we refer to \cite{LNP}. 
\end{Remark}

By setting
\begin{equation}\label{bht1.0}
R_0^2 := \frac{\lambda_1^2+\lambda_2^2}{\mu^2}+\frac{2C}{\mu}, \quad
Q = \left(\frac{\lambda_2}{\mu}\,,-\frac{\lambda_1}{\mu}\right)
\end{equation}
and
\begin{equation}\label{bht5.3}
M(s):=\left( \int_0^s \cos(\theta(t))\ dt, \ \int_0^s \sin(\theta(t))\ dt   \right)\in \partial \Omega,
\end{equation}
we can write \eqref{first-opt} as
\begin{equation}\label{bht1.1}
k(s)=\frac{\mu}{2}\left(  R_0^2 - \left\|\overrightarrow{QM(s)} \right\|^2 \right)^- \quad \forall s\in [0,2\pi].
\end{equation}
In particular, if $k(s)>0$, then
\begin{equation}\label{bht1.2}
k(s)=\frac{\mu}{2}\left( \left\|\overrightarrow{QM(s)} \right\|^2- R_0^2   \right).
\end{equation}

Using the tools of shape derivative, we can also write the optimality condition for the curvature $k$ of an optimal domain
in a different way.
\begin{Proposition}\label{P01}
Let $k$ be the curvature associated to an optimal domain $\Omega$ solution of \eqref{jmin}.
 Then, on the strictly convex parts it holds
 \begin{equation}\label{ve0.1}
k''(s)=-\frac{1}{2} k^3 - \lambda  k + \mu,
\end{equation}
where
\begin{equation}\label{bht5.4}
\lambda:=\frac{2\mu A(\Omega) -E(\Omega)}{2\pi}
\end{equation}
Moreover, using the notations introduced in \eqref{bht1.0} and \eqref{bht5.3}, and denoting by $\nn$ 
the unit normal exterior vector to $\Omega$, it holds at the point $M(s)$
\begin{equation}\label{first-opt-geom}
\langle\overrightarrow{QM}, \nn\rangle =\frac{\lambda}{\mu} + \frac{1}{2\mu}\,k^2 \quad \text{in}\ \partial \Omega.
\end{equation}
\end{Proposition}


\begin{Remark}
One can wonder whether a relation like $\langle\overrightarrow{QM}, \nn\rangle = a+b k^2$ implies that
the domain is a disk. According to Andrews, \cite[Theorem 1.5]{and}, this is certainly true if $a\leq 0$ and $b\geq 0$, since it is possible to prove that the isoperimetric ratio $P^2/A$ decreases
under a flow driven by such a relation.
As we will see below (Sections \ref{secE} and \ref{secnumeric}) this is not true in general
if both coefficients $a$ and $b$ are positive.
\end{Remark}

The proof of Proposition \ref{P01} makes the use of shape derivatives. For the reader convenience we present in the following lemma
the shape derivative of the area, of the perimeter and of the elastic energy. We postpone the proof to the Appendix.
\begin{Lemma}\label{Lem-shape}
The shape derivatives of the three quantities $A,P,E$ are given by
\begin{gather*}
dA(\Omega;V)=\int_{\partial\Omega} \langle V, \nn\rangle \,ds,\\
dP(\Omega;V)=\int_{\partial\Omega} k \langle V, \nn\rangle \,ds,\\
dE(\Omega;V)=-\int_{\partial\Omega} (k''+\frac{1}{2} k^3) \langle V, \nn\rangle \,ds.
\end{gather*}
where $V$ is any deformation field and $\nn$ the exterior normal vector.
\end{Lemma}


We are now in position to prove Proposition \ref{P01}.
\begin{proof}[Proof of  Proposition \ref{P01}]
From Lemma \ref{Lem-shape}, we deduce that for any admissible $V$, a solution $\Omega$ of \eqref{jmin} satisfies
\begin{equation}\label{bht0.6}
-\int_{\partial\Omega} (k''+\frac{1}{2} k^3) \langle V, \nn\rangle \,ds + \mu \int_{\partial\Omega}\langle V, \nn\rangle \,ds=
{\lambda} \int_{\partial\Omega} k \langle V, \nn\rangle\,ds,
\end{equation}
for some Lagrange multiplier ${\lambda}$ associated to the perimeter constraint.
In particular, on any part of $\partial \Omega$ where the domain is strictly convex,
\begin{equation}\label{bht0.5}
 \theta'''(s)=k''(s)=-\frac{1}{2} k^3 - {\lambda} k + \mu.
\end{equation}

On the other hand, on any part of $\partial \Omega$ where the domain is strictly convex
we can differentiate the optimality condition \eqref{opcs} obtaining
\begin{equation}\label{bht0.4}
 \theta'''(s) = \mu+\mu \left[\left(x(s)-\frac{\lambda_2}{\mu}\right)x''(s)+\left(y(s)+\frac{\lambda_1}{\mu}\right)y''(s)\right].
\end{equation}

Combining the above relation with \eqref{bht0.5}, we deduce that on the part of
$\partial \Omega$ where $k>0$ a.e. (namely on a strictly
convex part):
\begin{equation}
  \label{bht0.7}
\mu \;\langle\overrightarrow{QM}, \nn\rangle = {\lambda} + \frac{1}{2}\,k^2,
\end{equation}
where $Q$ denotes the point defined by \eqref{bht1.0}
and the exterior normal vector is $\nn=(\sin \theta,-\cos \theta).$

By continuity the above relation still holds true on segments of $\partial \Omega$ and it  writes
\begin{equation}
  \label{bht0.8}
\langle\overrightarrow{QM}, \nn\rangle = \frac{{\lambda}}{\mu}.
\end{equation}
Therefore for an optimal domain condition \eqref{bht0.7} holds true on the whole boundary.
Integrating \eqref{bht0.7}  on $\partial \Omega$, we obtain
$$
2\mu A(\Omega) = 2\pi {\lambda}  + E(\Omega).
$$
This proves the proposition.
\end{proof}

\begin{Remark}\label{remPer}
Let us focus on Equation \eqref{ve0.1}. 
This differential equation has a central role in the analysis of $\partial\Omega$ as it can be seen in the proof of the next proposition.

Notice that it can be explicitly solved by quadrature using Jacobian elliptic functions as shown in Section \ref{secnumeric}, Lemma \ref{lemmacn}.
\end{Remark}

\begin{Proposition}\label{P02}
Assume that $\Omega$ is an optimal domain. Then $\partial \Omega$ is periodic and is the union of suitably rotated and translated copies of a symmetric curve.

In particular, in the strictly convex parts, it holds
\begin{equation}\label{bht1.4}
\overrightarrow{QM}(s)= \left(\frac{\lambda}{\mu} + \frac{1}{2\mu}\,k(s)^2\right)\nn(s)+\frac{1}{\mu} k'(s) \tau(s).
\end{equation}
If the boundary of $\Omega$ contains segments, then they have the same length 
$$
L=2\sqrt{R_0^2-\left(\frac{\lambda}{\mu}\right)^2}.
$$
\end{Proposition}
\begin{proof}
First of all, let us consider the strictly convex case. Let us assume that $k(s)$ attains its maximum $k_M$ at $s=0$ and its minimum $k_m$ at $s=s_1$. Using equation \eqref{ve0.1} and the Cauchy--Lipschitz Theorem, one can deduce that $k$ is symmetric with respect to $s_1$ and attains $k_M$ also at $s=2s_1$. Using again the Cauchy--Lipschitz Theorem we find
$k(s)=k(s+2s_1)$ for $s\in [0,2s_1]$ that concludes the first part of the proposition.
%
%
%

In the case where there is a segment, we assume again that $k(s)$ attains its maximum $k_M$ at $s=0$. Let us call $b$ the first positive zero of $k$ so that $k>0$ in the interval $[0,b)$. Assume $k(s)=0$ for $s\in [b,b+L]$ and
$k(s)>0$ for $s$ in a right neighbourhood of $s>b+L$.
Using \eqref{bht1.2} and the continuity of $k(s)$, we deduce
\begin{equation}\label{bht5.5}
\left\|\overrightarrow{QM(b)}\right\|=R_0.
\end{equation}
Using \eqref{first-opt-geom} we obtain
\begin{equation}\label{bht5.6}
\langle\overrightarrow{QM(b)}, \nn(b)\rangle =\frac{\lambda}{\mu}.
\end{equation}
As a consequence,
$$
\left| \langle\overrightarrow{QM(b)}, \tau(b)\rangle \right|=\sqrt{R_0^2-\left(\frac{\lambda}{\mu}\right)^2}.
$$

Differentiating \eqref{bht1.2} in the strictly convex part $s\in (0,b)$, we deduce 
\begin{equation}\label{bht5.7}
k'(s)=\mu\; \langle\overrightarrow{QM(s)}, \tau(s)\rangle.
\end{equation}
Thus
\begin{equation}\label{bht5.7b}
\lim_{s\to b^-} k'=-\mu \sqrt{R_0^2-\left(\frac{\lambda}{\mu}\right)^2}.
\end{equation}
Following the above calculation, we can show 
$$
\lim_{s\to (b+L)^+} k'=\mu \sqrt{R_0^2-\left(\frac{\lambda}{\mu}\right)^2}.
$$
Using the Cauchy--Lipschitz Theorem, we deduce that $k(s+b+L)=k(b-s)$ for $s\in [0,b]$. That is the boundary of $\Omega$ is composed by symmetric curves with a segment of length $L$. To conclude it remains to estimate $L$. 
Equations \eqref{bht5.5} and   \eqref{bht5.6} entail $(L/2)^2=R_0^2-\left(\lambda/\mu\right)^2.$

In the strictly convex part, we combine \eqref{bht5.7} and \eqref{first-opt-geom}
to obtain \eqref{bht1.4}.
\end{proof}

Here below, we present the proof of Theorem \ref{Ththeta}.
\begin{proof}[Proof of Theorem \ref{Ththeta}]
The function $\theta(s)$ corresponds to a solution of \eqref{jmin} if and only if it is a solution of
\begin{equation}\label{pb_min_theta}
\inf_{\theta\in \mathcal{M}} j_\mu(\theta).
\end{equation}

Using classical theory for this kind of optimization problem with constraints in a Banach space (see, for instance,
Theorem 3.2 and Theorem 3.3 in \cite{MaurerZowe}),
we can derive the optimality conditions.
More precisely, let us introduce the closed convex cone $K$ of
$L^2(0,2\pi)\times \mathbb{R}^3$ defined by
$$
K:=L^2_+(0,{2\pi})\times \{(0,0,0)\},
$$
where
$$
L^2_+(0,{2\pi}):= \left\{\ell \in L^2(0,2\pi) \  ; \
  \ell \geq 0 \right\}.
$$
We also set for $\theta\in W^{1,2}(0,2\pi)$
$$
m(\theta)=\left(\theta', \quad \int_0^{2\pi} \cos(\theta(s)) \ ds, \quad \int_0^{2\pi}  \sin(\theta(s)) \ ds,
  \quad
\theta({2\pi})-\theta(0)-2\pi
\right).
$$
Then, Problem \eqref{pb_min_theta} can be written as
$$
\inf \left\{j_\mu(\theta), \quad \theta\in W^{1,2}(0,{2\pi} ), \quad m(\theta)\in K\right\}.
$$

As a consequence, for a solution $\theta$ of \eqref{pb_min_theta}
there exist $\ell\in L^2_+(0,{2\pi} )$, $(\lambda_1,\lambda_2,\lambda_3)\in\mathbb{R}^3$
such that the two following conditions hold:
\begin{eqnarray*}
j_\mu'(\theta)(v)=
\langle(\ell,\lambda_1,\lambda_2,\lambda_3),m'(\theta)(v)\rangle_{L^2(0,2\pi)\times
  \mathbb{R}^3} \quad \forall v\in W^{1,2}(0,2\pi),\\
\left((\ell,\lambda_1,\lambda_2,\lambda_3),m(\theta)\right)_{L^2(0,{2\pi} )\times
  \mathbb{R}^3} =0.
  \end{eqnarray*}

The two above conditions can be written as
\begin{multline}
  \label{opt1}
   \int_0^{2\pi}  \theta' v' \ ds
+  \frac{\mu}{2}\int_0^{2\pi} \int_0^s \cos(\theta(s)-\theta(t)) (v(s)-v(t))
 \ ds\ dt
\\
= \int_0^{2\pi}   \ell v' \ ds
-\lambda_1 \int_0^{2\pi}  \sin(\theta) v \ ds
+\lambda_2 \int_0^{2\pi}  \cos(\theta) v \ ds
+\lambda_3(v({2\pi} )-v(0)),
\end{multline}
\begin{equation}
  \label{opt2}
  \int_0^{2\pi}  \ell \theta' \ ds=0.
\end{equation}

In \eqref{opt1}, we have used that, due to the constraints of $\theta$ in \eqref{bht5.2} we have
\begin{multline}\label{eq10}
\int\int_T \sin(\theta(s))\cos(\theta(t))\,
ds\,dt = \int_0^{2\pi} \left(\int_0^s \cos(\theta(t)) \ dt\right)\sin(\theta(s)) \ ds=\\
- \int_0^{2\pi} \left(\int_0^s \sin(\theta(t)) \ dt\right)\cos(\theta(s)) \ ds
=\frac 12 \int\int_T \sin(\theta(s)-\theta(t))\,
ds\,dt.
\end{multline}
Standard calculation gives
$$
\int_0^{2\pi} \int_0^s \cos(\theta(s)-\theta(t)) (v(s)-v(t)) \
dt \ ds
=2 \int_0^{2\pi}
\left(\int_0^s \cos(\theta(s)-\theta(t)) \ dt
\right)
v(s) \ ds.
$$
We thus define
\begin{equation}\label{14:41}
f(s)=
\mu \int_0^s \cos(\theta(s)-\theta(t)) \ dt
+\lambda_1 \sin(\theta(s))-\lambda_2 \cos(\theta(s))\quad \mbox{for } s \in [0,2\pi]
\end{equation}
and we rewrite \eqref{opt1} as
\begin{equation}
   \int_0^{2\pi}  \theta' v' \ ds
+  \int_0^{2\pi}  f v \ ds
= \int_0^{2\pi}   \ell v' \ ds
\qquad v\in W^{1,2}_0(0,{2\pi} ).\label{opt1bis}
\end{equation}
Let us consider the continuous function $F\in W^{1,\infty}(0,2\pi)$ defined by
\begin{equation}\label{1.1}
F(s):=-\int_0^s f(\alpha) \ d\alpha.
\end{equation}
Then integrating by parts in \eqref{opt1bis} yields (for some constant $C$)
\begin{equation}
 \theta' = -F + \ell-C \quad \text{in} \ (0,{2\pi} ).\label{opt3}
\end{equation}
The above equation implies that 
$$
\ell -F-C \geq 0 \quad \text{in} \ (0,{2\pi} ).
$$
On the other hand condition \eqref{opt2} yields
$\ell \theta' =0 \quad \text{in} \ (0,{2\pi} )$ which implies
$\ell(\ell-F-C) =  0 \quad \text{in} \ (0,{2\pi} )$, thanks to relation \eqref{opt3}.

We rewrite the above equality by using the decomposition
$F+C=g^+-g^-$, (where $g^+$ and $g^-$ are the positive and negative parts of $F+C$):
\begin{eqnarray*}
\ell(\ell-F+C)=(\ell-g^++g^+)(\ell-g^++g^-) = (\ell-g^+)^2 +
g^-(\ell-g^+)+\\
+g^+(\ell-g^+)+g^+g^-
=(\ell-g^+)^2 + g^- \ell + g^+ (\ell-F-C)
\end{eqnarray*}
which is the sum of three non-negative terms. Thus
\begin{equation}\label{1.2}
\ell=(F+C)^+
\end{equation}
and in particular, from \eqref{opt3},
\begin{equation}\label{15:18}
 \theta' = (F+C)^- \quad \text{in} \ (0,{2\pi} )
\end{equation}
We deduce that $\theta\in W^{2,\infty}(0,2\pi)$.

\medskip
Using \eqref{xy} and \eqref{ci}, we can write
\begin{equation}\label{1.3}
\int_0^s \cos(\theta(s)-\theta(t)) \ dt=x(s) x'(s) + y(s) y'(s).
\end{equation}
Moreover, the function $F$ defined by \eqref{1.1} and \eqref{14:41} can be rewritten as
$$
F(s) = - \int_{0}^s \left[\mu \int_0^\alpha \cos(\theta(\alpha)-\theta(t)) \ dt + \lambda_1 \sin(\theta(\alpha)) -\lambda_2 \cos(\theta(\alpha)) \right] \ d \alpha
$$
and combining this relation with \eqref{1.3}, we obtain
\begin{multline*}
F(s) = - \int_{0}^s \left[\mu(x(\alpha) x'(\alpha) + y(\alpha) y'(\alpha)) + \lambda_1 y'(\alpha) -
\lambda_2 x'(\alpha) \right] \ d \alpha\\
=\frac{\mu}{2} \left(\frac{\lambda_1^2+\lambda_2^2}{\mu^2}- \left[x(s)-\frac{\lambda_2}{\mu}\right]^2-\left[y(s)+
\frac{\lambda_1}{\mu}\right]^2 \right).
\end{multline*}
The above relation and \eqref{15:18} yield
\eqref{first-opt}.
\end{proof}


\section{Geometric properties}\label{secgeom}
\subsection{Symmetries}
Using a classical {\it reflexion} method, we can prove that there always exists a minimizer
with a central symmetry:
\begin{Theorem}\label{T01}
There exists at least one minimizer of Problem \eqref{pb2} that has a center of symmetry.
\end{Theorem}

\begin{Remark}
Notice that the symmetry result in Proposition \ref{P02}  does not imply the centrally symmetric result, as it can be seen considering a smooth approximation of an equilateral triangle.
\end{Remark}

\begin{proof}[Proof of Theorem \ref{T01}]
Let $m_\mu$ denotes the value of the minimum of problem \eqref{pb2}:
$m_\mu:=\min_{\Omega \in \mathcal{A}} E(\Omega)+\mu A(\Omega)$ and let us consider a minimizer
$\Omega$, which exists thanks to Theorem \ref{teoexi}. For any direction (unit vector) $\eta$ let us denote by $X(\eta)$ a point on the boundary
of $\Omega$ whose exterior normal vector is $\eta$.
By a continuity argument (change $\eta$ in $-\eta$), there exists at least one direction
$\eta$ such that the segment
joining $X(\eta)$ to $X(-\eta)$ (if not unique, choose one) cuts the boundary in two parts
$\Gamma_+(\eta)$ and $\Gamma_-(\eta)$ having the same length $\pi$.
Let us denote by $\Omega_+(\eta)$ (resp. $\Omega_-(\eta)$) the part of $\Omega$ bounded
by the segment $[X(\eta),X(-\eta)]$ and $\Gamma_+(\eta)$ (resp. $\Gamma_-(\eta)$), see
Figure \ref{figOeta} and Figure \ref{figOplus-minus}.


\newcommand{\insieme}{(-6,0) to[out=270,in=180] (1,-4) to[out=0, in=270] (6,0) to[out=90, in=0] (4,5) to[out=180, in=90] (-6,0)}
\newcommand{\oominus}{(-6,0) to[out=270,in=180] (1,-4) to[out=0, in=270] (6,0)-- (6,1)}
\newcommand{\ooplus}{(6,1) to[out=90, in=0] (4,5) to[out=180, in=90] (-6,0)}

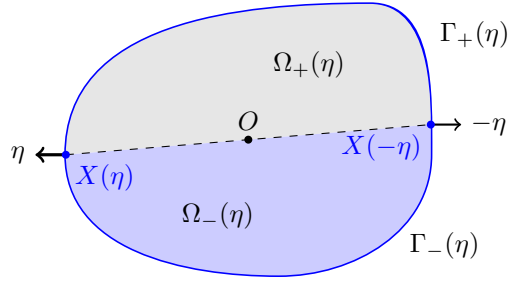
\begin{figure}[h]
\centering
\begin{tikzpicture}[x=4mm,y=4mm]
\draw[very thick, blue] \insieme; \draw (6,4)node[right]{$\Gamma_+(\eta)$};\draw (5,-3)[right]node{$\Gamma_-(\eta)$};
\draw[very thick,->] (-6,0)--(-7,0) node[left]{$\eta$};\draw[thick,->] (6,1)--(7,1)node[right]{$-\eta$};
\fill[gray!20] \ooplus;\draw (2,3)node{$\Omega_+(\eta)$};
\fill[blue!20] \oominus;\draw(-1,-2)node{$\Omega_-(\eta)$};
\draw[dashed] (-6,0)--(6,1);
\fill (0,0.5)node[above]{$O$} circle (1.5pt);\fill[blue] (-6,0)node[below right]{$X(\eta)$} circle (1.5pt);\fill[blue] (6,1)node[below left]{$X(-\eta)$} circle (1.5pt);
\end{tikzpicture}
\caption{The set $\Omega$ separated into the sets $\Omega_+(\eta)$ and $\Omega_-(\eta)$.}\label{figOeta}
\end{figure}

\begin{figure}[h]
\centering
\begin{tikzpicture}[x=4mm,y=4mm]
\draw (-6.3,2)node{$\Omega_+$};
\draw[very thick] \ooplus;
\fill[gray!20] \ooplus;\draw (3,3)node{$\Omega_+(\eta)$};
\begin{scope}[yshift=15, rotate=180]
\draw[very thick] \ooplus;
\fill[gray!20] \ooplus;\draw (-3,2)node{$\sigma(\Omega_+(\eta))$};
\end{scope}
\fill (0,0.5)node[above]{$O$} circle (1.5pt);\draw[dashed] (-6,0)--(6,1);
\draw[very thin, blue] \insieme;
\begin{scope}[xshift=8cm]
\draw (-6.5,2)node{$\Omega_-$};
\draw[very thick] \oominus;
\fill[blue!20] \oominus;\draw (3,-1)node{$\Omega_-(\eta)$};
\begin{scope}[yshift=11, rotate=180]
\draw[very thick] \oominus;
\fill[blue!20] \oominus;\draw (-3,-1)node{$\sigma(\Omega_-(\eta))$};
\end{scope}
\fill (0,0.5)node[above]{$O$} circle (1.5pt);\draw[dashed] (-6,0)--(6,1);
\draw[very thin, blue] \insieme;
\end{scope}
\end{tikzpicture}
\caption{The centrally symmetric sets $\Omega_+$ and $\Omega_-$.}\label{figOplus-minus}
\end{figure}
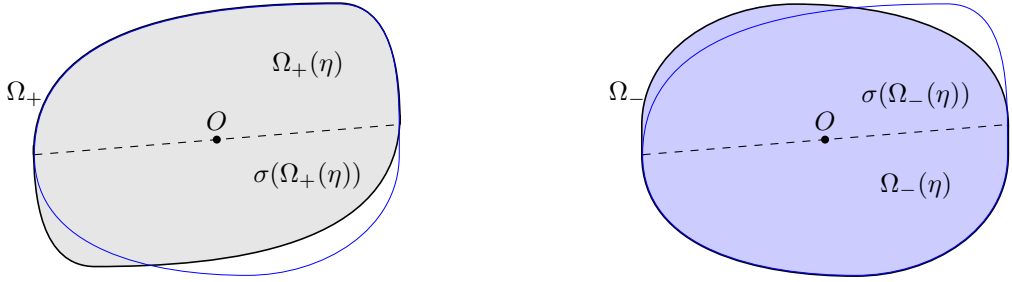

Let us denote by $O$
the middle of the segment $[X(\eta),X(-\eta)]$ and let $\sigma$ be the central symmetry with respect
to $O$; define
$$\Omega_+:=\Omega_+(\eta) \cup \sigma(\Omega_+(\eta)) \quad\mbox{and}\quad
\Omega_-:=\Omega_-(\eta) \cup \sigma(\Omega_-(\eta)).$$
By construction $\Omega_+$ and $\Omega_-$ have perimeter $2\pi$ and then they are admissible.
It follows
\begin{eqnarray*}\label{sym1}
E(\Omega_+)+\mu A(\Omega_+) \geq m_\mu, \\ 
E(\Omega_-)+\mu A(\Omega_-) \geq m_\mu\;.
\end{eqnarray*}
Adding these two inequalities yields
$$2m_\mu = 2 E(\Omega)+ 2 \mu A(\Omega) \geq 2 m_\mu.$$
Therefore, we have equality everywhere and both $\Omega_+$ and $\Omega_-$ solve the minimization problem.
Moreover they are centrally symmetric (and coincide if $\Omega$ is itself centrally symmetric).
\end{proof}

\begin{Remark}\label{axial}
We emphasize that, thanks to Proposition \ref{P02}, $\Omega$ is locally axially symmetric. Indeed the boundary of $\Omega$ can be decomposed in the union of suitably rotated and translated copies $\gamma_i$ of a symmetric curve. In particular, for each of these copies $\gamma_i$, there exists a point $M_i\in \gamma_i$  such that the curve $\gamma_i$ is symmetric with respect to the line $QM_i$.
\end{Remark}

\begin{Proposition}\label{axial2}
Assume that $\Omega$ is a centrally symmetric minimizer of Problem \eqref{pb2} and locally symmetric with respect to  lines passing by $Q$.
Then $\Omega$ is axially symmetric.

More precisely, up to translation the boundary of $\Omega$ can be decomposed in the union of suitable rotated copies $\gamma_i$ of a symmetric curve $\gamma$, that is 
\begin{equation}\label{bdomega}
\partial\Omega=\cup_{i=1}^p \gamma_i\qquad\text{with}\qquad\gamma_i(s)=\rho_i\gamma(s),\qquad\text{for}\qquad s\in[0,2s_1],
\end{equation}
where $\rho_i$ are rotations of the plane and $2p\, s_1=2\pi$. Moreover $\Omega$ has $p$  axis of symmetry.
\end{Proposition}

\begin{proof}
Thanks to Proposition \ref{P02}, $\Omega$ is locally axially symmetric and $\partial\Omega$ can be decomposed in the union of suitable rotated and translated copies $\gamma_i$ of a symmetric curve $\gamma$, that is $\partial\Omega=\cup_{i=1}^p (\rho_i\gamma(s)+b_i)$, for planar rotations $\rho_i$ and vectors $b_i$. 
In particular, for each of these copies $\gamma_i$, there exists a point $M_i\in \gamma_i$  such that the curve $\gamma_i$ is symmetric with respect to the line $QM_i$; $M_i=\gamma(s_1)$. 
Up to translations we can assume that $Q$ coincides with the origin. 

Fix $i\in\{1,...,p\}$ and let us consider the curve $\gamma_i$. 
Since $\Omega$ is centrally symmetric and $\gamma$ is symmetric with respect to the point $\gamma(s_1)$ there exists a corresponding index $l\in\{1,...,p\}$ such that the curves $\gamma_i$ and $\gamma_l$ are axially symmetric with respect to the line $\gamma_i(s_1)\gamma_l(s_1)$, through the origin (we can assume this line to be the axis $\{x=0\}$). 
Since this property holds true for each $j\in\{1,...,p\}$, we have that the points $\gamma_j(0),\gamma_j(2s_1)$ belongs to a common circle of radius $\|\gamma_1(0)\|$.
This entails that the decomposition in (\ref{bdomega}) holds true.

More precisely by the symmetry of the curve $\gamma$ (and hence that of $\gamma_j$), the curves $\gamma_{i-1}$ and $\gamma_{i+1}$ are axially symmetric too, since it holds $\gamma_{i-1}(2s_1)=\gamma_i(0)=(x_i(0),y_i(0))=(-x_i(2s_1),y_i(2s_1))=(-x_{i+1}(0),y_{i+1}(0))$, where $(x_j(s),y_j(s))$ denotes the point $\gamma_j(s)$ and $\gamma_{i-1}(0)$,$\gamma_{i+1}(2s_1)$ both belong to the common circle.
\end{proof}

\subsection{Segments}
We are interested in the analysis of existence of segments for minimizers of Problem \eqref{pb2}, i.e. non-empty intervals $(a,b)$ of $[0,2\pi]$ such that
$k(s)=0$ on $(a,b)$.
\begin{Lemma}\label{lem1}
Assume that $\Omega$ is a centrally symmetric minimizer of Problem \eqref{pb2}.
If $\partial \Omega$ has at least one segment, then
\begin{equation}\label{ve9.9}
E(\Omega)\leq \mu A(\Omega) \leq 2E(\Omega).
\end{equation}
\end{Lemma}
\begin{proof}
Assume that $\partial \Omega$ has at least one segment. Since $\Omega$ has a center of symmetry, $\partial \Omega$ has at least two parallel segments hence
$\Omega$ is contained in the infinite strip corresponding to the two segments. By \eqref{first-opt-geom}, the width of this strip is
$$
\frac{2\mu A(\Omega) -E(\Omega)}{\pi\mu}.
$$

As a consequence,
$$
A(\Omega)\leq \diam(\Omega) \frac{2\mu A(\Omega) -E(\Omega)}{\pi\mu}\leq 2A(\Omega) -\frac{E(\Omega)}{\mu},
$$
which implies the first inequality in \eqref{ve9.9}.

In order to prove the second inequality, we make a perturbation $\Omega_\varepsilon$ of $\Omega$. First we increase the size of a segment by $\varepsilon$ and we modify in a symmetric way the opposite segment.
Then we perform an homothety of center $Q$ and ratio $1/(1+\varepsilon/\pi)$ so that the perimeter of  $\Omega_\varepsilon$  remains equal to $2\pi$.

The domain $\Omega_\varepsilon$ satisfies
\begin{multline*}
E(\Omega_\varepsilon) +\mu A(\Omega_\varepsilon)=E(\Omega)\left(1+\frac{\varepsilon}{\pi}\right)
+\mu \left( A(\Omega)+\frac{2\mu A(\Omega) -E(\Omega)}{\pi\mu}\varepsilon \right)\frac{1}{\left(1+\dfrac{\varepsilon}{\pi}\right)^2}
\\
=E(\Omega) +\mu A(\Omega)+\frac{\varepsilon^2}{\pi^2} \left(2 E(\Omega) - \mu A(\Omega)\right) + o(\varepsilon^2).
\end{multline*}
This ends the proof since $E(\Omega_\varepsilon) +\mu A(\Omega_\varepsilon)\ge E(\Omega) +\mu A(\Omega).$
\end{proof}

\begin{Lemma}\label{Lem18:48}
For any $\mu>1$, the following inequalities hold for any optimal domain  $\Omega$.
\begin{equation}\label{bht18:35}
2\pi \sqrt{\mu} \leq E(\Omega)+\mu A(\Omega) \leq 3\pi \sqrt{\mu}-\pi.
\end{equation}
\end{Lemma}
\begin{proof}
Thanks to Theorem \ref{main-th1} by Gage \cite{Gage},
$$
E(\Omega)A(\Omega)\geq \pi^2.
$$
On the other hand relation \eqref{minper} entails a lower bound for the elastic energy: $E(\Omega)\geq \pi$.
As a consequence
$$
E(\Omega)+\mu A(\Omega)\geq E(\Omega)+\mu \frac{\pi^2}{E(\Omega)}\geq 2\pi \sqrt{\mu}
$$
by using that $\mu>1$.

To prove the second inequality, we consider the admissible stadium $\Omega_S$ composed by a rectangle of lengths $2/\sqrt{\mu}$ and $\pi(1-1/\sqrt{\mu})$ and by two half disks of radius
$1/\sqrt{\mu}$. For this stadium, $P(\Omega_S)=2\pi,$ and
$$
E(\Omega_S)+\mu A(\Omega_S)=\pi \sqrt{\mu}+\mu \left( \frac{\pi}{\mu} + \frac{2}{\sqrt{\mu}} \pi (1-\frac{1}{\sqrt{\mu}} ) \right)
=3\pi \sqrt{\mu}-\pi.
$$
\end{proof}

\begin{Theorem}\label{Thm0712}
Assume that $\Omega$ is a minimizer of Problem \eqref{pb2} with a center of symmetry.
Then $\partial \Omega$ contains either $0$ or $2$ segments.
\end{Theorem}
\begin{proof}
Assume that the boundary of $\Omega$ contains $m$ segments, hence thanks to Proposition \ref{P02}, $m=2N$ and suppose $m\ge4$ (that is $N\geq 2$). 
By Theorem \ref{T01} and Proposition \ref{axial2} the set $\Omega$ is contained in the union of $2N$ copies of isosceles triangles with common vertex in $Q$, height equal to $\lambda /\mu$, angle at the vertex equal to $ {\pi}/N$, moreover by its convexity $\Omega$ is contained in a regular $2N$-gon of inradius $\lambda /\mu$.

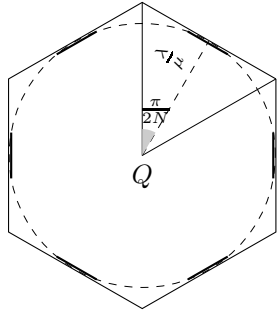
\begin{figure}[h]
\centering
\begin{tikzpicture}[x=0.35mm,y=0.35mm]
\draw[very thin, dashed] (0,0) circle (50);
\draw[thick] (50:50)--(70:50) ;\draw[thick] (110:50)--(130:50);\draw[thick] (170:50)--(190:50);\draw[thick] (230:50)--(250:50);\draw[thick] (290:50)--(310:50);\draw[thick] (350:50)--(370:50);
\draw[thin] (30:58)--(90:58)--(150:58)--(210:58)--(270:58)--(330:58)--(30:58);
\draw[very thin] (0,0)node[below]{$Q$}--(30:58);\draw[very thin] (0,0)--(90:58);
\draw[very thin, dashed] (0,0)--(60:50)node[near end,sloped, above]{\small$\frac{\lambda}{\mu}$};
\fill[gray!40] (0,0)--(60:10) arc(60:90:10)--cycle;\draw (60:10)node[above]{\small$\frac{\pi}{2N}$};
\end{tikzpicture}
\caption{The set $\Omega$ is contained in a regular $2N$-gon of inradius $\lambda /\mu$}
\end{figure}


Hence, comparing the perimeters, we deduce that 

\begin{equation}\label{bht18:33}
2\pi \leq \frac{\lambda}{\mu} 4N \tan\left( \frac{\pi}{2N}\right).
\end{equation}

On the other hand, we deduce from \eqref{bht18:35} and \eqref{ve9.9}
\begin{equation}\label{bht1.6}
E(\Omega)\geq \frac{2\pi \sqrt{\mu}}{3}.
\end{equation}
Combining again  \eqref{bht18:35} and \eqref{ve9.9} we obtain
\begin{equation}\label{bht2.7}
\frac{3}{2}\mu A(\Omega)\leq E(\Omega)+\mu A(\Omega)\leq 3\pi \sqrt{\mu}-\pi,
\end{equation}
and thus
\begin{equation}\label{bht1.7}
\mu A(\Omega)\leq 2\pi \sqrt{\mu}-\frac 23 \pi.
\end{equation}
Relations \eqref{bht1.6} and \eqref{bht1.7} imply
\begin{equation}\label{bht3.0}
\frac{\lambda}{\mu}=\frac{2\mu A(\Omega) -E(\Omega)}{2\pi\mu}\leq \frac{5}{3\sqrt{\mu}} -\frac{2}{3\mu}.
\end{equation}

On the other hand, since $E(\Omega)\geq \pi$ and $A(\Omega)\leq \pi$,
\begin{equation}\label{bht9.0}
\frac{\lambda}{\mu}=\frac{2\mu A(\Omega) -E(\Omega)}{2\pi\mu}\leq 1-\frac{1}{2\mu}.
\end{equation}

Consequently, by \eqref{bht18:33} and \eqref{bht3.0}
\begin{equation}\label{bht2.0}
\mu \leq \left(\frac{10 N}{3\pi} \tan\left( \frac{\pi}{2N}\right)\right)^2
\end{equation}
and by \eqref{bht18:33} and \eqref{bht9.0}
\begin{equation}\label{bht2.1}
\frac{1}{2\left(1- \dfrac{\pi}{2N} \cotan\left( \dfrac{\pi}{2N}\right)\right) } \leq  \mu.
\end{equation}
We can notice that the sequences
$$
\left\{ \left(\frac{10n}{3\pi} \tan\left( \frac{\pi}{2n}\right)\right)^2 \right\}_{n\geq 2}, \qquad
\left\{ \frac{1}{2\left(1- \dfrac{\pi}{2n} \cotan\left( \dfrac{\pi}{2n}\right)\right) } \right\}_{n\geq 2}
$$
are decreasing and increasing respectively and since for $n=3$,
$$
\left(\frac{10}{\pi} \tan\left( \frac{\pi}{6}\right)\right)^2 < \frac{1}{2\left(1- \dfrac{\pi}{6} \cotan\left( \dfrac{\pi}{6}\right)\right) }
$$
we deduce from \eqref{bht2.0} and from \eqref{bht2.1} that $N\leq 2$.

We also deduce that if $N=2$, then
$$
2.3\leq \mu \leq 4.6.
$$

If $N=2$, then we deduce that
\begin{equation}\label{plus0.0}
A(\Omega)\leq \left( 2\frac{\lambda}{\mu} \right)^2.
\end{equation}
Using \eqref{bht3.0}, we deduce 
\begin{equation}\label{plus0.0}
A(\Omega)\leq 4 \left(   \frac{5}{3\sqrt{\mu}} -\frac{2}{3\mu} \right)^2.
\end{equation}
The above relation and \eqref{bht1.6}  imply
\begin{equation}\label{bht3.0plus}
\frac{\lambda}{\mu}=\frac{2\mu A(\Omega) -E(\Omega)}{2\pi\mu}\leq  \frac{4}{\pi} \left(   \frac{5}{3\sqrt{\mu}} -\frac{2}{3\mu} \right)^2-\frac{1}{3\sqrt{\mu}}.
\end{equation}
Since \eqref{bht18:33} writes
$$
\frac{\lambda}{\mu} \geq \frac{\pi}{4}
$$
we deduce from \eqref{bht3.0plus} that
\begin{equation}\label{plus0.1}
\frac{\pi}{4}\leq \frac{4}{\pi} \left(   \frac{5}{3\sqrt{\mu}} -\frac{2}{3\mu} \right)^2-\frac{1}{3\sqrt{\mu}}.
\end{equation}
That yields 
$$
\mu<2.
$$

\end{proof}


\begin{Lemma}
Assume that $\Omega$ is a minimizer of Problem \eqref{pb2}.
 For $\mu\geq 1$, we have
\begin{equation}\label{minl0}
\lambda \geq \frac{\sqrt{1+16\mu}-1}{4} - \sqrt{\frac{\mu}{2}},
\end{equation}
and
\begin{equation}\label{bht4.1}
A(\Omega) \geq \frac{\pi}{4\mu}\left(\sqrt{1+16\mu} - 1\right).
\end{equation}
\end{Lemma}

\begin{proof}
Multiplying the optimality condition \eqref{first-opt-geom} by $k$ and integrating on
the boundary yields
\begin{equation}\label{minl1}
2\pi\mu = \mu \int_0^{2\pi} k \langle\overrightarrow{QM}, \nn\rangle \, ds = 2\pi \lambda + \frac{1}{2}\int_0^{2\pi} k^3\,ds.
\end{equation}
The Cauchy-Schwarz inequality and the previous equality (\ref{minl1}) give
\begin{equation}\label{minl2}
(2E(\Omega))^2=\left(\int_0^{2\pi} k^2 ds\right)^2\leq \int_0^{2\pi} k \, ds\; \int_0^{2\pi} k^3 \, ds =
8\pi^2(\mu - \lambda)\leq 8\pi^2\mu,
\end{equation}
the last inequality coming from the fact that ${\lambda}$ is necessarily positive if $\mu\geq 1$. Therefore, from (\ref{minl2}) we have
\begin{equation}\label{minl3}
E(\Omega)\leq \pi\sqrt{2\mu}.
\end{equation}
We use the Green-Osher inequality, valid for any (regular) convex domain, see \cite{Gre-Osh} and we obtain
\begin{equation}\label{minl4}
\int_0^{2\pi} k^3\,ds \geq \frac{P(\Omega)^2\pi -2A(\Omega)\pi^2}{A(\Omega)^2}.
\end{equation}
Plugging \eqref{minl4} into \eqref{minl1}, we obtain
\begin{equation}\label{bht4.0}
4\pi(\mu-\lambda)=\int_0^{2\pi} k^3\,ds\geq \frac{4\pi^3 -2A(\Omega)\pi^2}{A(\Omega)^2}
\end{equation}
and hence
\begin{equation}\label{minl5}
\mu A(\Omega)^2 + \frac{\pi}{2} A(\Omega) -\pi^2 \geq (\mu - \lambda) A(\Omega)^2 + \frac{\pi}{2} A(\Omega) -\pi^2 \geq 0.
\end{equation}
Considering the sign of the polynomial, this implies
\begin{equation}\label{minl6}
\mu A(\Omega) \geq \frac{\pi}{4}\left(\sqrt{1+16\mu} - 1\right).
\end{equation}
The proof concludes by using \eqref{bht5.4}, (\ref{minl3}) and (\ref{minl6}).
\end{proof}

Let us now prove that, for sufficiently large $\mu$, the optimal domains are not strictly convex.
\begin{Proposition}\label{segmu}
If $\mu > 47.7750$, then the boundary of an optimal domain $\Omega$ contains segments.
\end{Proposition}
\begin{proof}
Let us multiply \eqref{bht0.5} by $k'$:
\begin{equation}\label{bht2.4}
\frac{(k')^2}{2}=-\frac{k^4}{8}-\frac{{\lambda}}{2}k^2 + \mu k +C.
\end{equation}
where $\lambda$ is defined in \eqref{bht5.4} and
\begin{equation}\label{bht2.5}
C=\frac{k_M^4}{8}+\frac{{\lambda}}{2}k_M^2 - \mu k_M.
\end{equation}
where $k_M>0$ is the maximum of $k$. 

Notice that equation \eqref{bht2.4} can be written as
$$
\frac{(k')^2}{2}=\mathscr{P}(k),
$$
where $\mathscr{P}$ is a concave polynomial function (using the fact that $\lambda>0$). As a consequence, either $\mathscr{P}$ has 2 distinct roots $k_m<k_M$
or $\mathscr{P}$ has a double root $k_m=k_M$ (and $\mathscr{P}\leq 0$).
Therefore $\partial \Omega$ has a segment if and only if $k_m<0$. Note that this condition is equivalent to $\mathscr{P}(0)=C>0$.

Relation \eqref{bht0.7} evaluated at the point $s_M$ such that $k(s_M)=k_M$, entails 
\begin{equation}\label{bht3.7}
\|\overrightarrow{QM(s_M)}\|=\frac{{\lambda}}{\mu}+ \frac{1}{2\mu} k_M^2.
\end{equation}
Notice that $\|\overrightarrow{QM(s_M)}\|\geq 1$ otherwise the whole domain will be included in the disk of center $Q$ and of radius one (because $\|\overrightarrow{QM(s_M)}\|$ is the radius of the circumscribed disk according to (\ref{bht1.2}), 
which leads to contradiction since the set $\Omega$ has perimeter $2\pi$.
Therefore, we deduce
\begin{equation}\label{bht6.1}
k_M^2\geq 2\mu - 2{\lambda}.
\end{equation}
Let us denote by $E_\mu$ and $A_\mu$ the elastic energy and the area of an optimal domain.
We use inequality (\ref{bht18:35}) and  Gage's inequality \eqref{main_ineq} to obtain
$$
\frac{\pi^2}{A_\mu} + \mu A_\mu \leq E_\mu + \mu A_\mu \leq 3\pi \sqrt{\mu} - \pi.
$$
Assuming $\mu>1$, we deduce from the above estimate that
\begin{equation}\label{bht6.0}
A_\mu \leq \mathcal{A}_M(\mu):=\frac{\pi}{2\mu}\left( (3\sqrt{\mu}-1)+\sqrt{5\mu-6\sqrt{\mu}+1} \right).
\end{equation}
The above estimate, the definition of $\lambda$ in (\ref{bht5.4}) and Gage's inequality \eqref{main_ineq}  yield
\begin{equation}\label{seg1}
\lambda=\frac{2\mu A_\mu - E_\mu}{2\pi}\,\leq \lambda_M(\mu):=\frac{2\mu \mathcal{A}_M(\mu)-\pi^2/\mathcal{A}_M(\mu)}{2\pi}.
\end{equation}
Therefore \eqref{bht6.1} implies
\begin{equation}\label{seg2}
k_M^2\geq 2\mu - 2{\lambda}_M(\mu).
\end{equation}
Since $\mu > {\lambda}_M(\mu)$ for $\mu >3$, we deduce 
\begin{equation}\label{seg2.1}
k_M^4\geq (2\mu - 2{\lambda}_M(\mu))^2.
\end{equation}

At last, we use the fact that the half-diameter of the optimal set is less than $\pi/2$
(because the perimeter is $2\pi$) and the optimality condition (\ref{bht3.7}) to  get
$$\frac{{\lambda}}{\mu}+ \frac{1}{2\mu} k_M^2=
\|\overrightarrow{QM(s_M)}\|\leq \frac{\pi}{2}\,.$$
Combining the above estimate with \eqref{minl0}, it follows
\begin{equation}\label{seg3}
k_M(\mu)\leq \sqrt{\pi \mu -2\lambda_m(\mu)},
\end{equation}
where
$$
\lambda_m(\mu):=\frac{\sqrt{1+16\mu}-1}{4} - \sqrt{\frac{\mu}{2}}.
$$

Gathering \eqref{seg2.1}, \eqref{minl0}, \eqref{seg2}, and \eqref{seg3}, we deduce from  the definition of the constant $C$ in
\eqref{bht2.5} that
$$
C\geq \frac{1}{8} (2\mu - 2{\lambda}_M(\mu))^2 
+ \frac{1}{2}\lambda_m(\mu) (2\mu - 2{\lambda}_M(\mu))
- \mu \sqrt{\pi \mu -2\lambda_m(\mu)}.
$$
It turns out that the function of $\mu$ in the right-hand side is positive as soon as $\mu>47.775$
which proves the result.
\end{proof}

\section{The disk}\label{secdisk}
As already pointed out, Gage's inequality (Theorem \ref{main-th1}) asserts that the disk minimizes the product $E(\Omega)A(\Omega)$ among convex bodies with given perimeter. This leads to the following result.
\begin{Corollary}
The disk is the unique minimizer to Problem \eqref{pb2} for $\mu\leq 1$.
\end{Corollary}
\begin{proof}
Let $\Omega$ be a convex set of perimeter $2\pi$ and let $D$ be the unit disk. 
Using Gage's inequality \eqref{main_ineq} we have
$$
E(\Omega)+A(\Omega)\geq 2\sqrt{E(\Omega)A(\Omega)} \geq 2\sqrt{E(D)A(D)} = E(D)+A(D),
$$
the last equality coming from the fact that $E(D)=A(D)=\pi$. 
Therefore the disk is the minimizer to Problem \eqref{pb2}  for $\mu=1$.
For $\mu\leq 1$, we use the isoperimetric inequality for the elastic energy, expressed in \eqref{minper}, \eqref{pb3}, to obtain:
\begin{eqnarray*}
E(\Omega)+\mu A(\Omega)=\mu(E(\Omega)+A(\Omega))+(1-\mu)E(\Omega)\geq \\
 \mu(E(D)+A(D))+(1-\mu)E(D) = E(D)+\mu A(D).
\end{eqnarray*}
\end{proof}

Notice that the disk cannot be the solution for large $\mu$. 
Indeed considering the stadium $\Omega_S$
obtained as the union of a rectangle of length $\pi/2$ with two half-disks of radius $1/2$ one gets $P(\Omega_S)=2\pi$ and
\begin{equation}\label{19:03}
E(\Omega_S)+\mu A(\Omega_S) = 2\pi +\mu \frac{3\pi}{4}.
\end{equation}
Comparing this with the value of $E(D)+\mu A(D)$ (where $D$ is the unit disk) we obtain an equality for $\mu=4$ while \eqref{19:03} 
gives a strictly better value for $\mu>4$.
We are going to show that in fact the disk cannot be the solution for $\mu>3$ since it is no longer a local minimum. 
Notice that the value $\mu=3$ is probably optimal since the numerical algorithm presented in Section \ref{secnumeric}) seems to show that the disk is optimal for $\mu\leq 3$.
\begin{Theorem}\label{theo-disk}
The unit disk $D$ is a local strict minimum for Problem \eqref{pb2} if and only if $\mu\leq 3$.
\end{Theorem}
\begin{proof}
We consider small perturbations of the unit disk obtained through perturbations of its
support function. First, let $\mu>3$ and consider the convex body $\Omega_\ep$ whose support function is
\begin{equation}\label{loc1}
h_\ep(t):=1+\ep \cos 2t.
\end{equation}
Notice that the set $\Omega_\ep$ is $C^2_+$ while $\ep< 1/3$. Moreover, using \eqref{paramphi}, its area is
$$
A(\Omega_\ep)=\frac{1}{2}\int_0^{2\pi} (1+\ep \cos(2t))(1-3\ep \cos(2t))\,dt=\pi(1-\frac{3}{2}\,\ep^2),
$$
and its elastic energy is
$$
E(\Omega_\ep)=\frac{1}{2}\int_0^{2\pi} \frac{dt}{1-3\ep\cos(2t)}=\frac{\pi}{\sqrt{1-9\ep^2}}.
$$
Therefore
$$
J_\mu(\Omega_\ep)=E(\Omega_\ep) +\mu A(\Omega_\ep)=\pi(1+\mu)-\pi \frac{\ep^2}{2}\,(3\mu-9)+o(\ep^2),
$$
which is strictly less than $J_\mu(D)=E(D)+\mu A(D)=(1+\mu)\pi$ for $\ep$ small enough.

\smallskip
Conversely for $\mu\leq 3$ let $D_\ep$ be a perturbation of the unit disk with $D_\ep\in C^2_+$ and
let $d_\ep$ denote its support function. 
Since $P(\Omega_\ep)=2\pi$ and from \eqref{eq:sturm}, we can write
\begin{equation}\label{loc}
d_\ep(t)=1+\ep \sum_{k=2}^{+\infty} a_k\cos(kt)+b_k\sin(kt).
\end{equation}
Using again \eqref{paramphi} and since $d_\ep\in C^2$, we have
$$
A(D_\ep)=\pi\left(1-\frac{\ep^2}{2}\sum_{k=2}^{+\infty} (k^2-1)[a_k^2+b_k^2]\right),
$$
and
$$
E(D_\ep)=\pi\left(1+\frac{\ep^2}{2}\sum_{k=2}^{+\infty} (k^2-1)^2[a_k^2+b_k^2] +o(\ep^2)\right).
$$
Thus
$$
J_\mu(D_\ep)-J_\mu(D)=\frac{\pi\ep^2}{2}\sum_{k=2}^{+\infty} \left((k^2-1)^2-\mu(k^2-1))[a_k^2+b_k^2] +o(\ep^2\right),
$$
which is positive for $\ep$ small enough when
either $\mu<3$, whatever the $a_k,b_k$ are, or $\mu=3$ if at least one of the $a_k,b_k$ are non-zero for $k\geq 3$.

It remains to consider the case $\mu=3$, $a_k=b_k=0$ for $k\geq 3$ for which a direct computation gives
$$J_\mu(D_\ep)-J_\mu(D)=\frac{3^5\pi\ep^4}{2^3}\,[a_2^2+b_2^2]^2+o(\ep^4)$$
and the result follows.

Finally, for general perturbation not necessarily $C^2_+$, we use Theorem \ref{thmdensity}.
\end{proof}

\section{Description of the Blaschke-Santal\'o diagram}\label{secE}
We want to study the set
\begin{equation}\label{setE1}
\mathcal{E}:=\left\{(x,y)\in \R^2, x=\frac{4\pi A(\Omega)}{P(\Omega)^2},
y=\frac{E(\Omega)P(\Omega)}{2\pi^2},\,\Omega\in \mathcal{C} \right\},
\end{equation}
where $\mathcal{C}$ is defined by \eqref{bht5.0}.

Notice that by homogeneity the sets $\Omega$ and $t\Omega$ correspond to the same point in $\E$.
Therefore, without loss of generality we can consider convex sets with fixed perimeter $P(\Omega)=2\pi$.

According to the classical isoperimetric inequality and
inequality (\ref{minper}), for any $(x,y)\in \mathcal{E}$ it holds
$x\leq 1$ and $y\geq 1$.
Moreover, we emphasize that Gage's inequality (\ref{main_ineq}) writes as $xy\geq 1$.

Let us now present the main result of this section: a convexity result for the set $\E$.
\begin{Theorem}
The set $\E$ is 
convex with respect to both the horizontal and the vertical directions. That is: 
{if $(x_0,y_0)\in\E$, then $[x_0,1)\times [y_0,\infty)\in \mathcal{E}$.}

Moreover, the half-line $(x=1, y=s)$, for $s\in [1,+\infty)$ is contained in the boundary of $\mathcal{E}$ and it is not included in $\mathcal{E}$ except for the point $(1,1)$. 
\end{Theorem}
\begin{proof}
Let us first prove 
the convexity result for the regularized set
\begin{equation}\label{setE1reg}
\mathcal{E}_{reg}:=\left\{(x,y)\in \R^2, x=\frac{4\pi A(\Omega)}{P(\Omega)^2},
y=\frac{E(\Omega)P(\Omega)}{2\pi^2},\,\Omega\in \mathcal{C}_{reg} \right\},
\end{equation}
where $\mathcal{C}_{reg}$ is the subset of $\mathcal{C}$ of convex bodies $\Omega$ whose radius of curvature $\phi=h+h''$ is positive and of class $C^1$.

{We show that for any $(x_0,y_0)\in\E_{reg}$, { the segment } $(x=t,y=y_0)$,  { is contained in } $\E_{reg},$ { for }$t\in [x_0,1)$; and the half-line  $(x=x_0,y=s)$, { is contained in } $\E_{reg}$ for $s\in [y_0,+\infty)$.
}

We first show the vertical convexity. 
Let us take $(x_0,y_0)\in \mathcal{C}_{reg}$ corresponding to a convex set $\Omega$ of perimeter $2\pi$.
Without loss of generality (up to rotations), we can assume that 
$$
\xi:=\min_{t\in [0,2\pi]} \phi >0,
$$
is attained at $t=0$. 
Let us assume that $\Omega$ is not the unit disk.
Since $P(\Omega)=2\pi$ and by condition \eqref{eq:sturm}, we can write
$$
h(t)=1+\sum_{n\geq 2} \alpha_n \cos(n t)+\beta_n \sin(nt) \quad \forall t\in [0,2\pi].
$$ 
Therefore
$$
\phi(0)=1-\sum_{n\geq 2} \alpha_n (n^2-1)=\xi,
$$
with $\xi<1$ thanks to the expression in (\ref{paramphi}) and to $P(\Omega)=2\pi$.
This implies that there exists $m\geq 2$ such that
$$\alpha_m=\frac{1}{\pi}\, \int_0^{2\pi} h(t)\cos mt\,dt > 0.$$

Let us introduce the convex set $\Omega_1$ defined through its support function by 
$$
h_1(t)=h(t)+a_m\cos(mt),
$$
where $a_m$ is a suitable constant such that $|a_m|<\mu/(m^2-1)$. We emphasize that this guarantees the convexity of $\Omega_1$ since $h_1+{h_1}''>0$. 
In particular the perimeter of $\Om_1$ is $2\pi$
and its area is given by
$$
A(\Om_1)=\frac{1}{2}\int_0^{2\pi} h_1 (h_1 + {h_1}'')=A(\Om)+\frac{\pi}{2} (1-m^2)[a_m^2+2a_m\alpha_m].
$$
Notice that $A(\Om_1)<A(\Om)$ if $0<a_m<\mu/(m^2-1)$. 

Let us now denote by $I$ the interval $I=(0,\mu/(m^2-1))$.
By the isoperimetric inequality $A(\Omega)<\pi$, thus we can choose an integer $p\not= m$, $p\geq 2$ such that
\begin{equation}\label{vc1}
    \frac{\pi}{2(p^2 -1)}\,\leq \pi - A(\Omega).
\end{equation}
Let $a_p$ be a real number satisfying $|a_p|<1/(p^2-1)$.
We introduce a new convex set $\Omega_0$ through its support function: $h_0(t)=1+a_p\cos(pt)$.
By construction we have $\phi_0:={h_0}''+h_0=1+(1-p^2)a_p \cos(pt)>0$, $\Om_0$ has perimeter $2\pi$
and its area is given by $A(\Om_0)=\pi-\pi(p^2-1)a_p^2/2$.
By assumption (\ref{vc1}), we have $A(\Om_0)>A(\Om)$ for any $a_p\in [0,\frac{1}{p^2-1})$.
We denote by $J$ the interval $J=[0,\frac{1}{p^2-1})$.

\smallskip
For $\tau\in [0,1]$ let us consider
the Minkowski combination $\Om_\tau:=\tau\Omega_1 +(1-\tau)\Omega_0$ whose support function is $h_\tau=\tau h_1+(1-\tau)h_0$,
see \cite{schneider} for additional properties of the Minkowski sum.
We have $P(\Omega_\tau)=2\pi$ and its area is given by
\begin{multline}\label{vc2}
  A(\Om_\tau)=\frac{1}{2}\int_0^{2\pi} h_\tau (h_\tau + {h_\tau}'') \ dt
  =\tau^2 \left(A(\Om)
  +\frac{\pi}{2}(1-m^2)[a_m^2+2a_m\alpha_m]\right) \\
  +(1-\tau)^2\left(\pi+\frac{\pi}{2}(1-p^2) a_p^2\right)+\tau(1-\tau)(1-p^2)  \pi a_p\alpha_p.
\end{multline}
Notice that the right-hand side of formula (\ref{vc2}) defines a continuous (quadratic) function of $\tau$, say $g(\tau;a_m,a_p)$
such that,
$$\forall a_m,a_p\in I\times J, \;g(0;a_m,a_p)=A(\Om_0)>A(\Om)\;\mbox{and}\;g(1;a_m,a_p)=A(\Om_1)<A(\Om).$$ Therefore,
for any fixed $a_m, a_p$ in $I\times J$ there exists a value $\tau(a_m,a_p)\in [0,1]$ such that $A(\Om_\tau)=A(\Om)$.
Moreover, the function $(a_m,a_p)\mapsto \tau(a_m,a_p)$ can be chosen such that $\tau$ is continuous and $\tau(0,0)=1$.

The elastic energy of $\Om_\tau$ is given by
\begin{equation}\label{vc3}
    E(\Om_\tau)=\frac{1}{2}\,\int_0^{2\pi} \frac{dt}{(1-\tau)[1+(1-p^2)a_p\cos(pt)]+\tau[\phi+(1-m^2)a_m\cos(mt)]}\,.
\end{equation}
If we replace $\tau$ by $\tau(a_m,a_p)$ this expression defines a continuous function $E(a_m,a_p)$ of $a_m$ and $a_p$ such that
$E(0,0)=E(\Om)$. Moreover since the denominator of the quotient
vanishes at $t=0$ when $a_m$ approaches $\mu/(m^2-1)$ and $a_p$ approaches $1/(p^2-1)$ and since $\phi\in C^1[0,2\pi]),$
the Fatou lemma yields 
$$\lim_{\substack{a_m\to \mu/(m^2-1)\\ a_p\to 1/(p^2-1)}} E(a_m,a_p)=+\infty.$$
Thus the set of values taken by $E(a_m,a_p)$ when $(a_m,a_p)$ varies in $I\times J$ contains $[E(\Om), +\infty)$
which proves that the whole half line $(x_0,y), y\in [y_0,+\infty)$ is in the domain $\E$.


\medskip
We prove the horizontal convexity. Let $\Omega_0$ and $\Omega_1$ be two convex domains of perimeter $2\pi$
with the same elastic energy. We denote by $(x_0,y)$ and $(x_1,y)$ the corresponding points in $\mathcal{E}_{reg}$.
We claim that the elastic energy is convex for the Minkowski sum:
\begin{equation}\label{hc1}
\forall \tau \in [0,1],\ E((1-\tau) \Omega_0 +\tau \Omega_1)\leq (1-\tau) E(\Omega_0)+\tau E(\Omega_1).
\end{equation}
This can be proven by using \eqref{paramphi} and the convexity of the function $x\mapsto 1/x$. 

Let us now consider the  path in the diagram $\E$, joining the two points $(x_0,y)$ and $(x_1,y)$, obtained by the 
points corresponding to the convex
combination $(1-\tau) \Omega_0 +\tau \Omega_1$. Inequality (\ref{hc1}) implies that the whole path
is below the horizontal line of ordinate $y$. We conclude to the fact that all points $(x,y)$ with
$x\in [x_0,x_1]$ belong to $\mathcal{E}_{reg}$ using the vertical convexity.

Let us consider convex domains with a support function defined
by $h_{n,a}(t)=1+a \cos(nt)$, $n\in \mathbb{N}, n\geq 2$, $a\in\mathbb{R}, |a|<1/(n^2-1)$.
The area and elastic energy are given by
$$A(\Omega_{n,a})=\pi-\frac{\pi(n^2-1)a^2}{2} \qquad E(\Omega_{n,a})=\frac{\pi}{\sqrt{1-(n^2-1)^2a^2}}$$
This yields a family of parametric curves $x(a;n)=A(\Omega_{n,a})/\pi$ and $y(a;n)=E(\Omega_{n,a})/\pi$
which accumulate on the half-line $x=1, y\in [1,+\infty)$ when $n$ increases. Finally, this line does
not contain any point in $\mathcal{E}_{reg}$ because the isoperimetric inequality gives $x<1$ except for the ball
(among convex domains). Using the horizontal convexity, this also allows to prove that
for any $(x_0,y_0)\in \mathcal{E}_{reg}$ the horizontal segment
$x=t,y=y_0, t\in [x_0,1) \mbox{ is contained in } \mathcal{E}_{reg}.$

The proof concludes by using Theorem \ref{thmdensity} below.
More precisely, assume $(x_0,y_0)\in\E$. Let us consider $(x_0,y)$ with $y>y_0$. Then there exists $(x_\ep,y_\ep)\in \mathcal{E}_{reg}$ with $x_\ep<x_0$ and $y_\ep<y$.
Using the first part of the proof, we deduce that $(x_0,y)\in [x_\ep,1)\times [y_\ep,\infty)\subset \mathcal{E}_{reg}$.

Consider $(x,y_0)$ with $x>x_0$. We can use relation \eqref{hc1} which is valid for $\Omega\in \mathcal{C}$ (using Theorem \ref{thmdensity}) and follow the same proof as in the case of 
$\mathcal{E}_{reg}$.

%
\end{proof}

\begin{Theorem}\label{thmdensity}
Let $\Omega$ be a convex domain in the class $\mathcal{C}$ (defined in (\ref{bht5.0})),
then there exists a sequence of $C^\infty$ regular strictly convex domains $\Omega_\ep$
in the class $\mathcal{C}$ such that:
\begin{enumerate}
\item the support function $h_\ep$ of $\Omega_\ep$ is of class $C^\infty$;
\item the curvature of $\Omega_\ep$, $k_\ep$ satisfies $k_\ep(s)\geq \ep>0$;
\item the elastic energy converges: $\lim_{\ep\to 0}E(\Omega_\ep) = E(\Omega)$;
\item the area converges: $\lim_{\ep\to 0}A(\Omega_\ep) = A(\Omega)$  and we can assume $A(\Omega_\ep)<A(\Omega)$.
\end{enumerate}
\end{Theorem}
\begin{proof}
Since $\Omega$ is an open bounded convex set, there exist a point $O$ and
two positive numbers $0<r_0<R_0$ such that
$B(O,r_0) \subset \Omega \subset B(O,R_0)$. We use the gauge function $u$ which
defines the convex domain in polar coordinates $(r,\tau)$:
$$\Omega:=\left\{(r,\tau)\in (r_0,R_0)\times \R \ ; \ r<\frac{1}{u(\tau)}\right\}$$
where $u$ is a positive and $2\pi$-periodic function. Since $\partial\Omega$ is contained in the ring $B(O,R_0) \setminus B(O,r_0)$ we have
\begin{equation}\label{de1}
R_0^{-1}\leq u(\tau) \leq r_0^{-1}
\end{equation}
Moreover, it is classical that
the convexity of $\Omega$ is equivalent to the fact that $u''+u$ is a non-negative measure.
Let us detail the regularity of $u$ when $\Omega$ belongs to the class
$\mathcal{C}$ and the link between the gauge function $u$ and the support function $h$.
From the parametrization 
$$
\begin{cases}
x(\tau)=(\cos\tau)/u(\tau),\\
y(\tau)=(\sin\tau)/u(\tau)
\end{cases}
$$ 
we deduce
$ds=\sqrt{u(\tau)^2+{u'(\tau)}^2}/u(\tau)^2\,d\tau$ which makes the curvilinear abscissa $s$
an increasing function of the angle $\tau$. The unit tangent vector is given by
$$
\begin{cases}
\cos\theta = - \frac{(\sin\tau) u(\tau) + (\cos\tau) u'(\tau)}{\sqrt{u(\tau)^2+{u'(\tau)}^2}}\\
\sin\theta =  \frac{(\cos\tau) u(\tau) - (\sin\tau) u'(\tau)}{\sqrt{u(\tau)^2+{u'(\tau)}^2}},
\end{cases}
$$
while the exterior normal vector is $\nn=(\sin\theta, - \cos\theta)$. 
Since the boundary of $\Omega$ is strictly convex, there is a one-to-one correspondence between $\tau$ and $\theta$.
We write $h(\tau)$, meaning $h(\theta(\tau))$, given by
\begin{equation}\label{de2}
h(\tau)=\langle\overrightarrow{OM}, \nn\rangle=\frac{1}{\sqrt{u(\tau)^2+{u'(\tau)}^2}}.
\end{equation}
Since $r_0\leq h(\tau)\leq R_0$, we deduce from (\ref{de2}) that $u'\in L^\infty(\R)$. 
Moreover the 
curvature at the point $(x(\tau),y(\tau))$ is given by 
$$
k(\tau)=\frac{u^3}{(u^2+{u'}^2)^{3/2}}\,(u+u''),
$$
and hence the elastic energy is
\begin{equation}\label{de3}
    E(\Omega) = \frac{1}{2}\,\int_0^{2\pi} \frac{u^4(u+u'')^2\,d\tau}{(u^2+{u'}^2)^{5/2}} .
\end{equation}
Therefore, the fact that $\Omega$ belongs to the class
$\mathcal{C}$ means that the function 
$$
\tau\mapsto \frac{u^2(u+u'')}{(u^2+{u'}^2)^{5/4}}
$$ 
is in $L^2(0,2\pi)$.
More precisely $u\in W^{2,2}(0,2\pi)$, since $u$ and 
$u'$ are bounded and $u$ is bounded from below, and hence $u''$ is in
$L^2(0,2\pi)$. The converse is also true that is
\begin{equation}\label{de4}
    \Omega \in \mathcal{C} \Longleftrightarrow u\in W^{2,2}(0,2\pi).
\end{equation}
Let us assume that $\Omega$ is strictly convex, of class $C^\infty$, whose curvature is bounded from
below: i.e. $k\geq \ep>0$. Then its gauge function is of class $C^\infty$,
and the function $\tau\mapsto \theta(\tau)$ is strictly increasing, $C^\infty$ and its
derivative is given by 
$$
\frac{d\theta}{d\tau}= \frac{u(u+u'')}{u^2+{u'}^2}
$$ 
which is also bounded from below.
Therefore its inverse function $\theta\mapsto \tau(\theta)$ is of class $C^\infty$. Notice that the angle
of the normal vector $\nn$ with the $x$-axis is $t=\theta -\pi/2$, therefore $\tau\mapsto t(\tau)$ is also $C^\infty$.
Moreover the support function can be expressed in terms of $t$ by:
$$
h(t)=\frac{1}{\sqrt{u(t(\tau))^2+{u'(t(\tau))}^2}},
$$
which is $C^\infty$.

We now proceed to the approximation result. 
Let $\Omega$ be any convex domain in the class $\mathcal{C}$
and let $u$ be its gauge function.
We choose a sequence of (non-negative) $C^\infty$ mollifiers $\rho_\epp$ (of support of size $\epp$) and we consider the regularized functions $u_\ep$ defined by
$u_\ep:=u\ast \rho_{\ep^4} + a\ep$ where $a:=2R_0^3/r_0^3$. When $\ep$ goes to zero, the convolution
product $u\ast \rho_{\ep^4}$ converges to $u$ in $W^{2,2}(0,2\pi)$, therefore $u_\ep$ converges to $u$ in $W^{2,2}(0,2\pi)$
and up to subsequence, we can assume that $u_\ep$ and $u'_\ep$ converge uniformly to $u$ and $u'$, respectively
(by the compactness embedding $W^{2,2}\hookrightarrow C^1$).
In particular the sequence of convex domains $\Om_\ep$ defined by their gauge functions $u_\ep$ converges to $\Om$
in the Hausdorff metric and in particular $B(O,r_0/2^{1/6}) \subset \Omega_\ep \subset B(O,2^{1/6} R_0)$
for $\ep$ small enough.
Since $u_\ep''+u_\ep=(u''+u)\ast\rho_{\ep^4} + a\ep\geq a\ep$, the curvature $k_\ep$ of $\Om_\ep$ satisfies
$$k_\ep(\tau)=\frac{u_\ep^3}{(u_\ep^2+{u_\ep'}^2)^{3/2}}\,(u_\ep+u_\ep'')\geq R_0^{-3} 2^{-1/2} r_0^3 2^{-1/2} a\ep=\ep.$$
Following the previous discussion, this implies that the support function of $\Om_\ep$ is of class $C^\infty$.
The $L^2$ convergence of $u_\ep''$ to $u''$ and the uniform convergence of $u_\ep$ and $u'_\ep$ to $u$ and $u'$, respectively,
ensure that 
$$
E(\Omega_\ep) = \frac{1}{2}\,\int_0^{2\pi} \frac{u_\ep^4(u_\ep+u_\ep'')^2\,d\tau}{(u_\ep^2+{u_\ep'}^2)^{5/2}}
$$ 
converges
to $E(\Om)$.
The convergence of the area $A(\Omega_\ep)$ to $A(\Omega)$ follows from the expression of the area in terms of the gauge function:
\begin{equation}\label{aireu}
A(\Om)=\frac{1}{2}\,\int_0^{2\pi} \frac{d\tau}{u^2(\tau)},
\end{equation}
and the uniform convergence.
Let us consider the perimeter of $\Om_\ep$, given by 
\begin{equation}\label{perimeteru}
P(\Om_\ep)=\int_0^{2\pi} \sqrt{u_\ep^2+{u'_\ep}^2}\,d\tau ;
\end{equation}
it converges to $P(\Om)$. Hence it suffices to make an homothety of $\Om_\ep$ of ratio $P(\Om)/P(\Om_\ep)$ to construct
a sequence of convex domains 
$$
\widetilde{\Omega}_\ep=\frac{P(\Omega)}{P(\Omega_\ep)}\Omega_\ep,
$$ with fixed perimeter and which fulfills the same properties as $\Om_\ep$.

It remains to show that in the above construction $A(\widetilde{\Om}_\ep)<A(\Om)$. 
For this purpose 
we need to show that
\begin{equation}\label{tocheck}
A(\Om_\ep)<A(\Om) \quad \text{and} \quad P(\Om_\ep)>P(\Om).
\end{equation}
In order to obtain this, we note that
$$
u_\varepsilon(x)-u(x)=a\varepsilon + \int_0^{2\pi } \rho_{\ep^4} (y) \left[u(x-y)-u(x)\right]\ dy.
$$
Using the fact that the support of $\rho_{\ep^4}$ is of size $\ep^4$, we deduce that
$$
\left|\int_0^{2\pi } \rho_{\ep^4} (y) \left[u(x-y)-u(x)\right]\ dy\right|\leq \ep^4 \|u'\|_{L^\infty(0,2\pi)}.
$$
As a consequence, for $\ep$ small, 
$$
u_\varepsilon(x)>u(x),
$$
which yields $A(\Om_\ep)<A(\Om)$ by formula \eqref{aireu}.

On the other hand, 
$$
u_\varepsilon'(x)-u'(x)=\int_0^{2\pi } \rho_{\ep^4} (y) \left[u'(x-y)-u'(x)\right]\ dy.
$$
Using again the fact that the support of $\rho_{\ep^4}$ is of size $\ep^4$, we deduce that
$$
\left|\int_0^{2\pi } \rho_{\ep^4} (y) \left[u'(x-y)-u'(x)\right]\ dy\right|\leq 
\int_0^{2\pi } \rho_{\ep^4} (y) \left| \int_{x}^{x-y} u'' \ da\right|\ dy\leq 
\ep^2 \|u''\|_{L^2(0,2\pi)}.
$$
As a consequence, for $\ep$ small, 
$$
\sqrt{u_\ep^2+{u'_\ep}^2}>\sqrt{u^2+{u'}^2}
$$
which yields $P(\Om_\ep)>P(\Om)$ by formula \eqref{perimeteru}.

\end{proof}

\section{Numerical algorithm}\label{secnumeric}
In this section, we show some numerical results regarding the problem
\begin{equation}\label{num1}
    \min\Big\{E(\Om)+\mu A(\Om), \Om\in \mathcal{C}, P(\Om)=2\pi \Big\}
\end{equation}
where $\mathcal{C}$ is defined by \eqref{bht5.0}
and we apply it to plot the convex hull of the Blaschke-Santal\'o diagram
(\ref{setE1}).

To solve the optimization problem (\ref{num1}) we choose to directly consider the optimality
conditions (\ref{ve0.1}) in term of the curvature $k(s)$. More precisely, we consider
the ODE
\begin{equation}\label{num2}
    \left\lbrace
    \begin{array}{rcl}
k''&=&-\frac{1}{2}\,k^3 - \lambda k + \mu  \\
k(0)&=&k_M\\
k'(0)&=&0
\end{array}
\right.
\end{equation}
where 
 $\lambda$ is the Lagrange multiplier
defined in Proposition 2.4, see (\ref{bht5.4}), and $k_M$ is the maximum value of the curvature.
The first step of the numerical procedure consists in evaluating these two parameters $\lambda$ and $k_M$.
The ODE (\ref{num2}) being valid only on the strictly convex parts of the boundary, we have to decide
whether we look for a strictly convex solution (without segments) or for a solution with segments. According
to Proposition \ref{segmu}, we know that we have segments on the boundary when $\mu$ is large enough.
It turns out that such segments appear numerically as soon as $\mu\geq 3.34...$, that is when the function $k$
vanishes before $s=\pi/2q$ ($q$ being the periodicity).

\smallskip
Let us explain in detail the procedure. First we notice that we can obtain an explicit formula for the solution to \eqref{num2} in term of the elliptic Jacobi function $\cn$ (see, for instance, \cite[chapter 16]{AbSt}).

\begin{Lemma}\label{lemmacn}
Assume $k\not\equiv 1$ and $\lambda\geq 0$. Then the solution of \eqref{num2} can be written as
\begin{equation}\label{num5-tak}
    k(s)=\frac{\alpha \cn(\omega s|\tau^2) +\beta}{\gamma \cn(\omega s|\tau^2) + 1}
\end{equation}
\end{Lemma}
Let us note that this solution may not be the curvature of a domain in $\mathcal{C}$, since for large $s$ it may be negative.

\begin{proof}
For any given data $(k_M,\lambda)$ we integrate once
equation (\ref{num2}) to get
\begin{equation}\label{num3}
    (k')^2=- \frac{1}{4}\, k^4 - \lambda k^2 +2\mu k + C
\end{equation}
where $C= \frac{1}{4}\, k_M^4 + \lambda k_M^2 - 2\mu k_M$.
This shows that the solution of \eqref{num2} is global and bounded. Moreover it has a minimum value $k_m$ and thus the polynomial
$$
Q(z)= - \frac{1}{4}\, z^4 - \lambda z^2 +2\mu z + C
$$
has two real roots $k_M$ and $k_m$ and two conjugate non-real roots $z_0$ and $\overline{z_0}$.
We make a change of variables to transform \eqref{num3}. In order to do this we introduce
\begin{equation}\label{num3b-tak}
\sigma = \frac{k_M+k_m}{2},\ \delta= \frac{k_M-k_m}{2}.
\end{equation}
Using the relations between the coefficients and the roots of $Q$, we check that $\sigma, \delta>0$.

We can also verify that there exists a unique root $\gamma\in (-1,0)$ of
\begin{equation}\label{polygam}
X^2+\frac{1}{\sigma\delta}\,\left(3\sigma^2 + \delta^2+ 2\lambda \right) X +1 =0.
\end{equation}

We can then define
\begin{equation}\label{num3b}
\alpha = \gamma \sigma +\delta,\ \beta=\gamma\delta + \sigma
\end{equation}
and perform the change of variables
$$
k=\frac{\alpha y + \beta}{\gamma y + 1},
$$
{that is } $y=\frac{\beta-k}{\gamma k-\alpha}$.
Tedious calculation and \eqref{polygam}  yield that $y$ satisfies
\begin{equation}\label{num4}
    \left\lbrace
    \begin{array}{rcl}
  (y')^2&=&\omega^2(1-y^2)(1-\tau^2 +\tau^2 y^2)) \\
  y(0)&=&1\,.
\end{array}
\right.
\end{equation}
with
$$\omega^2=\sigma\delta \frac{\gamma^2 - 1}{2\gamma}>0\quad \tau^2= \frac{\gamma^2 +\frac{\delta\gamma}{2\sigma}}{\gamma^2-1}\in (0,1).$$

It is well-known (see, for instance, \cite[chapter 16]{AbSt}) that the solution of (\ref{num4}) is the Jacobian elliptic function $s\mapsto \cn(\omega s|\tau^2)$.
Therefore, we have obtained \eqref{num5-tak} with $\alpha, \beta, \gamma, \omega, \tau$ defined as above.
\end{proof}

According to Remark \ref{remPer} and Proposition \ref{P02} (see also Proposition \ref{axial2}), the curve is $2q$ periodic, for an integer $q$, $q\geq 1$ in the strictly convex case
and $1\leq q\leq 2$ when there are segments (see Theorem \ref{Thm0712}).
We will use two different algorithms when looking for strictly convex solutions, \textbf{case a)} and non-strictly convex solutions, \textbf{case b)}.

\medskip\noindent
\textbf{Case a):} {\it strictly convex solutions}:
We choose $q\geq 1$ and we try to find the parameters $k_M,\lambda$ such that the two following conditions are satisfied:
\begin{eqnarray}\label{num6}
  2 K(\tau^2) & = & \frac{\omega\pi}{2q} \quad\mbox{(periodicity)}\\
  \int_0^{\pi/2q} k(s)\,ds & = & \frac{\pi}{2q}  \quad\mbox{(in order to have $\theta(\pi/2q)=\pi/2q$)}
\end{eqnarray}
where $K$ is the complete Elliptic integral of the first kind
\begin{equation}\label{ellipticK}
K(m):=\int_{0}^{1} \frac{d t}{\sqrt{(1-t^2)(1-mt^2)}} \quad (m\in [0,1]).
\end{equation}
which defines the periodicity of the Jacobian
elliptic function $\cn$, see \cite[chapter 17]{AbSt}. This gives a $2\times 2$ non-linear system that can be solved by
using the Levenberg-Marquardt algorithm implemented in \verb?Matlab?. Then we compute the angle $\theta(s)$ by integrating $k(s)$
and the curve by
$$x(s)=\int_0^s \cos\theta(u)\,du,\quad y(s)=\int_0^s \sin\theta(u)\,du.$$
The elastic energy is computed by integrating (numerically) the curvature squared, while the
area of the domain is computed using (\ref{bht5.4}): $A(\Omega)=\frac{2\pi \lambda + E(\Omega)}{2\mu}$.

\medskip\noindent
\textbf{Case b):} {\it non-strictly convex solutions}:
Choose $q=1,2$ and consider the first zero of the function $k(s)$, named $s_1$. Hence  $s_1\le \pi/2q$. We
search numerically parameters $k_M,\lambda$ such that the two following conditions are satisfied:
\begin{eqnarray}\label{num6a}
  \int_0^{s_1} k(s)\,ds &=& \frac{\pi}{2q},\\ \label{612}  
  \frac{2}{\mu}\,\sqrt{C} + 2 s_1 &=& \frac{\pi}{q};\label{613}
\end{eqnarray}
notice that \eqref{612} guarantees $\theta(s_1)=\pi/2q$).
Indeed, according to Proposition \ref{P02} and equality (\ref{bht5.7b}), any segment have the same length $L$
which is related to the value of $k'(s_1^-)$ by $L=-\frac{2}{\mu}\,k'(s_1^-)$. Using equation (\ref{num3}) and the fact that $k(s_1)=0$
we see that $k'(s_1^-)=-\sqrt{C}$. Thus (\ref{613}) expresses the fact that the total
length of the curve has to be $2\pi$.
Then one can proceed following the same steps as above.

\medskip
This method gives the following results: it turns out that a value of $q$ greater than one is never competitive
(with respect to $q=1$),
neither in the case of strictly convex domains, nor in the case where segments appear; the disk remains the optimal
domain while $\mu\leq 3$, showing that Theorem \ref{theo-disk} is probably optimal. As already mentioned,
segments appear as soon as $\mu\geq 3.3425$.
Figure \ref{fig-shapes} shows three optimal domains obtained for $\mu=3.2, \mu=4$ and $\mu=8$.
\begin{figure}[h!]
\begin{center}
   \begin{minipage}[c]{.25\linewidth}
      \includegraphics[scale=0.4]{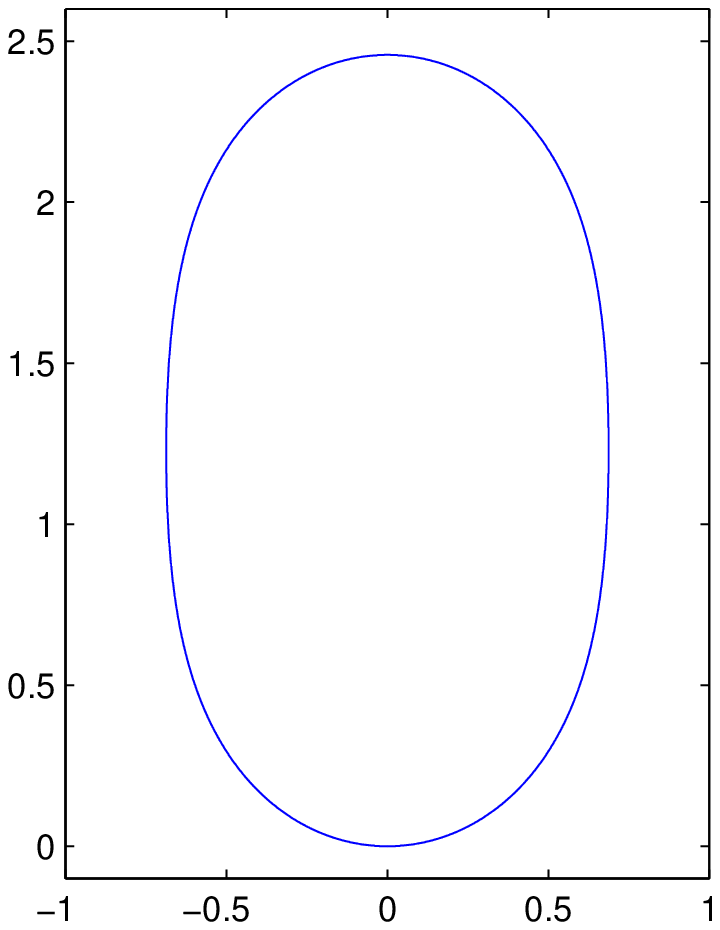}
   \end{minipage}
   \begin{minipage}[c]{.25\linewidth}
      \includegraphics[scale=0.4]{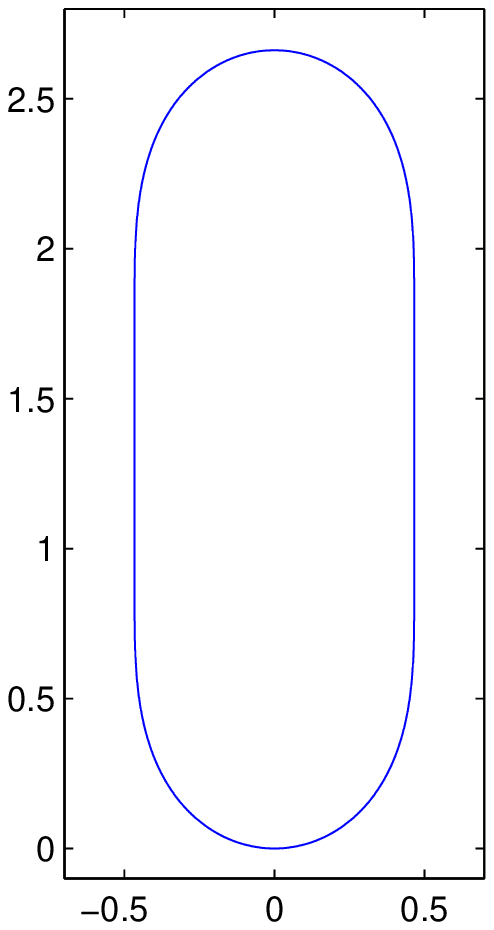}
   \end{minipage}
   \begin{minipage}[c]{.25\linewidth}
      \includegraphics[scale=0.4]{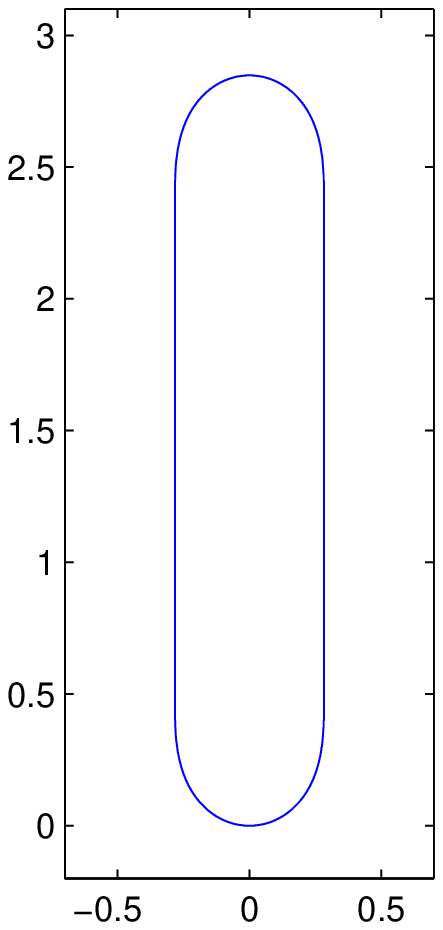}
   \end{minipage}
    \caption{Three optimal domains corresponding to $\mu=3.2, 4, 8$.}
       \label{fig-shapes}
    \end{center}
\end{figure}

\medskip
We are going to use the previous method to find optimal domains for any value $\mu$ starting at
$\mu=1$ in order to plot the unknown part of the boundary of the convex hull of the Blaschke-Santal\'o diagram $\mathcal{E}$, defined at (\ref{setE1}), which is contained in the half plane $\{x<1\}$.
This is equivalent to find the point(s) of $\mathcal{E}$ whose supporting line is parallel
to $y+\mu x =0$. Numerically, this process gives a unique continuous family of convex domains,
say $\Omega_\mu$, which tends to prove that the set $\mathcal{E}$ is indeed convex and can be plotted
this way. We show it in Figure \ref{figE}.
\begin{figure}[h!]
\begin{center}
      \includegraphics[scale=0.7]{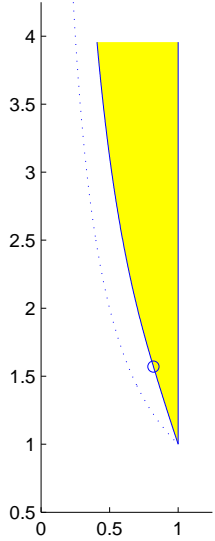}
    \caption{The Blaschke-Santalo diagram $\mathcal{E}$}
          \label{figE}
    \end{center}
\end{figure}

The lower point of $\mathcal{E}$ is obviously the disk whose coordinates are $(1,1)$ due to the normalization. 
The other point which appears on the boundary corresponds to the last strictly convex solution (obtained for $\mu=3.3425$). 
The dotted line is the hyperbola $yx=1$ which is the lower bound given by Gage's Theorem \ref{main-th1}. 
Actually, this is not the asymptotic hyperbola for the set $\mathcal{E}$. 
Next proposition makes the asymptotic behaviour of the set $\mathcal{E}$ on its left boundary more precise.
\begin{Proposition}
The hyperbola $yx=\rho^2/\pi^2$ is asymptotic to the set $\mathcal{E}$, where $\rho$ is given by
\begin{equation}\label{num8}
\rho=2\sqrt{2\pi}\left[2\FF\left(\frac{1}{2}\right) \,-K\left(\frac{1}{2}\right)\right]\,\simeq 4.2473
\end{equation}
where $K(\cdot)$ and $\FF(\cdot)$ are the complete Elliptic Integral of the first and the second kind, respectively:
$$
\FF\left(\frac{1}{2}\right)=\int_0^1 \sqrt{ \frac{1- t^2/2}{1-t^2} } \ dt, \quad K\left(\frac{1}{2}\right)=\int_{0}^{1} \frac{d t}{\sqrt{(1-t^2)(1-t^2/2)}} .
$$
More precisely, the elastic energy and the area of the optimal domains $\Omega_\mu$ behave, when $\mu \to +\infty$,
like
$$E(\Omega_\mu)\sim \rho \sqrt{\mu} \quad A(\Omega_\mu)\sim \rho/\sqrt{\mu}.$$
\end{Proposition}
\begin{proof}
Let us denote by $E_\mu=E(\Omega_\mu)$ and $A_\mu=A(\Omega_\mu)$ the elastic energy and the area
of an optimal domain. For $\mu$ large, the boundary of $\Omega_\mu$ contains segments (according
to Proposition \ref{segmu}), therefore, following Lemma \ref{lem1}, $E_\mu \leq A_\mu \leq 2E_\mu$.
We plug these inequalities into (\ref{bht18:35}) to get
$$\frac{\pi}{\sqrt{\mu}} \leq A_\mu \leq \frac{2\pi}{\sqrt{\mu}} \quad \mbox{and}\quad
\frac{2\pi\sqrt{\mu}}{3} \leq E_\mu \leq \frac{3\pi\sqrt{\mu}}{2}\,.$$
This shows that $\sqrt{\mu}A_\mu$ and $E_\mu/\sqrt{\mu}$ are bounded. Therefore, there exist
$\rho_1\in [\pi,2\pi]$ and $\rho_2\in [2\pi/3,3\pi/2]$ such that, up to some subsequence,
$\sqrt{\mu}A_\mu\to \rho_1$ and $E_\mu/\sqrt{\mu}\to \rho_2$. By (\ref{bht5.4}) it follows
\begin{equation}\label{num9}
\lambda\sim \frac{2\rho_1-\rho_2}{2\pi}\,\sqrt{\mu}\ \mbox{when $\mu\to +\infty$}.
\end{equation}
Since $A_\mu$ goes to 0 when $\mu$ goes to $+\infty$, the optimal (convex) domain $\Omega_\mu$
converges to a segment of length $\pi$ (and the half-diameter converges to $\pi/2$).
Let us denote by $H_M$ the point on the boundary which
is at maximum distance of $Q$. According to (\ref{bht1.2}), it corresponds to the point with
maximal curvature $k_M$. Therefore, using (\ref{first-opt})
$$
\frac{\lambda}{\mu}+\frac{1}{2\mu}k_M^2=\langle\overrightarrow{QH_M},\nn\rangle=\left\|\overrightarrow{QH_M}\right\|\to \frac{\pi}{2}.
$$
But $\lambda/\mu\to 0$ by (\ref{num9}), therefore
\begin{equation}\label{num10}
k_M\sim \sqrt{\pi \mu}\ \mbox{when $\mu\to +\infty$}.
\end{equation}
We now use the notation and formula in the proof of Lemma \ref{lemmacn}. By (\ref{num3}) and (\ref{num10}), we have $C\sim \pi^2\mu^2/4$
while the length of the segments satisfies $L\to \pi$. The optimal domain $\Omega_\mu$
contains a rectangle of length $L$ and width $2\lambda/\mu$ plus a part that can be included in a rectangle of edges sizes $1/2(\diam(\Omega_\mu)-L)$ and
$2\lambda/\mu$. Since $\diam(\Omega_\mu)\to \pi$, this remaining part is of order $o\left(1/\sqrt{\mu}\right)$
Therefore, using (\ref{num9}) and the definition of $\rho_1$ we have
$A_\mu\sim \rho_1/\sqrt{\mu}$ together with $A_\mu\sim 2L \lambda/\mu\sim (2\rho_1-\rho_2)/\sqrt{\mu}$.
It follows that $\rho_1=\rho_2$. We will now denote by $\rho$ this common value.

Straightforward calculations now give (keeping the above notations):
$$k_Mk_m\sim -\pi\mu,\ k_M+k_m\sim \frac{4}{\pi},\ \sigma\sim \frac{2}{\pi},\ \delta\sim \sqrt{\pi\mu},$$
and
$$\gamma\sim \frac{-2}{\pi\sqrt{\pi\mu}},\ \alpha\sim \sqrt{\pi\mu},\ \beta=o(1),\ \omega^2\sim \frac{\pi\mu}{2},
\ \tau^2=\frac{1}{2}+o(1).$$
Therefore, at order 1, the curvature behaves like
\begin{equation}
k(s)\sim \sqrt{\pi\mu}\; \cn\left(\omega s \mid \frac{1}{2}\right)\ \ \mbox{with} \ \omega=\sqrt{\frac{\pi\mu}{2}}.
\end{equation}
and $s_1$ the first zero of $k(s)$ tends to 0 as $K\left(\frac{1}{2}\right)/\omega$ since $K\left(\frac{1}{2}\right)$
is the first zero of $\cn\left(s\mid\frac{1}{2}\right)$ (the definition of $K$ is recalled in \eqref{ellipticK}).
It follows that the elastic energy satisfies
$$E_\mu= 2\int_0^{s_1} k^2(s) ds \sim 2\pi\mu \int_0^{s_1} \cn^2\left(\omega s \mid \frac{1}{2}\right)\, ds \sim
2\sqrt{2\pi\mu} \int_0^{K(1/2)} \cn^2\left(t \mid \frac{1}{2}\right) \, dt\,.$$
Now using formulae in  \cite[17.2.11]{AbSt} for Elliptic integral of second kind,
$$
\int_0^{K(1/2)} \cn^2\left(t \mid \frac{1}{2}\right)\, dt = 2 \FF\left(K\left(\frac{1}{2}\right) \mid \frac 12 \right)- K\left(\frac{1}{2}\right),
$$
where
$$
\FF(u\mid m) := \int_0^{\sn\left(t \mid m \right)} \sqrt{ \frac{1-m t^2}{1-t^2} } \ dt.
$$
Using $\sn(K(0.5)=1$ and formulae \cite[17.3.3]{AbSt} give the desired result.
Finally, since the accumulation point for $\sqrt{\mu}A_\mu$ and $E_\mu/\sqrt{\mu}$ are unique,
both sequences converge to $\rho$.
\end{proof}

\section{Appendix}
Here below we present the proof of Lemma \ref{Lem-shape}.
For the shape derivative formulas of the area and of the perimeter, we refer, for instance, to \cite{HP}.
 
In order to derive $E(\Omega)$ with respect to the domain, we consider a parametrization of $\partial \Omega$:
$$
s\in [0,P(\Omega)] \mapsto (x(s),y(s)).
$$
Let us consider a variation of the domain of the form $\Omega_{\varepsilon}=\Omega +
\varepsilon V (\Omega)$, where $V$ is a smooth function. Then a parametrization of
$\partial \Omega_{\varepsilon}$ is
\begin{gather*}
x_{\varepsilon}(s) =  x(s) + \varepsilon V_1(x(s),y(s)),\\
y_{\varepsilon}(s) =  y(s) + \varepsilon V_2(x(s),y(s)),
\end{gather*}
so that
\begin{gather*}
x_{\varepsilon}'(s) =  x'(s) + \varepsilon \frac{d}{ds}V_1(x(s),y(s))=\cos \theta(s)+ \varepsilon \frac{d}{ds}V_1(x(s),y(s)),\\
y_{\varepsilon}'(s) =  y'(s) + \varepsilon \frac{d}{ds}V_2(x(s),y(s))=\sin \theta(s)+ \varepsilon \frac{d}{ds}V_2(x(s),y(s)),
\end{gather*}
and
\begin{gather*}
x_{\varepsilon}''(s) = -\theta'(s) \sin \theta(s)+ \varepsilon \frac{d^2}{ds^2}V_1(x(s),y(s)),\\
y_{\varepsilon}''(s) =  \theta'(s) \cos \theta(s)+ \varepsilon \frac{d^2}{ds^2}V_2(x(s),y(s)).
\end{gather*}
We notice that
$$
x_{\varepsilon}'(s)^2+y_{\varepsilon}'(s)^2=1 + 2\varepsilon \left(\cos \theta(s)\frac{d}{ds}V_1(x(s),y(s)) + \sin \theta(s)\frac{d}{ds}V_2(x(s),y(s)) \right)  + o(\varepsilon^2).
$$
Moreover,
\begin{multline*}
k_\varepsilon(s) = \left[-x_{\varepsilon}''(s)y_{\varepsilon}'(s) + y_{\varepsilon}''(s)x_{\varepsilon}'(s)\right]/\left[x_{\varepsilon}'(s)^2+y_{\varepsilon}'(s)^2 \right]^{3/2} \\
=\left[ \left(\theta'(s) \sin \theta(s)- \varepsilon \frac{d^2}{ds^2}V_1(x(s),y(s))\right) \left(  \sin \theta(s)+ \varepsilon \frac{d}{ds}V_2(x(s),y(s))\right)\right.\\
+\left.\left(\theta'(s) \cos \theta(s)+ \varepsilon \frac{d^2}{ds^2}V_2(x(s),y(s)) \right) \left( \cos \theta(s)+ \varepsilon \frac{d}{ds}V_1(x(s),y(s))\right)\right]\\
\times \left(1 -3\varepsilon \left(\cos \theta(s)\frac{d}{ds}V_1(x(s),y(s)) + \sin \theta(s)\frac{d}{ds}V_2(x(s),y(s)) \right)\right)  + o(\varepsilon^2)
\end{multline*}
which yields
\begin{multline*}
k_\varepsilon(s) = \theta'(s)+\varepsilon\left( \theta'(s) \sin \theta(s) \frac{d}{ds}V_2(x(s),y(s))- \sin \theta(s) \frac{d^2}{ds^2}V_1(x(s),y(s))\right.\\
+\left.\theta'(s) \cos \theta(s)\frac{d}{ds}V_1(x(s),y(s)) +  \frac{d^2}{ds^2}V_2(x(s),y(s))\cos \theta(s)
\right)\\
-3\varepsilon \left(\cos \theta(s)\theta'(s)\frac{d}{ds}V_1(x(s),y(s)) + \theta'(s) \sin \theta(s)\frac{d}{ds}V_2(x(s),y(s)) \right)+ o(\varepsilon^2)
\end{multline*}
and thus
\begin{multline*}
k_\varepsilon(s) = \theta'(s)+\varepsilon\left( -2 \theta'(s) \sin \theta(s) \frac{d}{ds}V_2(x(s),y(s))- \sin \theta(s) \frac{d^2}{ds^2}V_1(x(s),y(s))\right.\\
\left.-2 \theta'(s) \cos \theta(s)\frac{d}{ds}V_1(x(s),y(s)) +  \frac{d^2}{ds^2}V_2(x(s),y(s))\cos \theta(s)
\right)+ o(\varepsilon^2).
\end{multline*}

Consequently, we can write the elastic energy for the perturbation of $\Omega$:
\begin{multline*}
E(\Omega_{\varepsilon})=\frac{1}{2} \int_0^{2\pi} \left[\theta'(s)+\varepsilon\left( -2 \theta'(s) \sin \theta(s) \frac{d}{ds}V_2(x(s),y(s))- \sin \theta(s) \frac{d^2}{ds^2}V_1(x(s),y(s))\right.\right.\\
\left.\left.-2 \theta'(s) \cos \theta(s)\frac{d}{ds}V_1(x(s),y(s)) +  \frac{d^2}{ds^2}V_2(x(s),y(s))\cos \theta(s)
\right)
+ o(\varepsilon^2)\right]^2 \\
\times \left[1 + 2\varepsilon \left(\cos \theta(s)\frac{d}{ds}V_1(x(s),y(s)) + \sin \theta(s)\frac{d}{ds}V_2(x(s),y(s)) \right)  + o(\varepsilon^2)\right]^{1/2} \ ds
\end{multline*}
Thus
\begin{multline*}
E(\Omega_{\varepsilon})=\frac{1}{2} \int_0^{2\pi} \theta'(s)^2+2\varepsilon\left( -2 \theta'(s)^2 \sin \theta(s) \frac{d}{ds}V_2(x(s),y(s))- \theta'(s)\sin \theta(s) \frac{d^2}{ds^2}V_1(x(s),y(s))\right.\\
\left.-2 \theta'(s)^2 \cos \theta(s)\frac{d}{ds}V_1(x(s),y(s)) +  \frac{d^2}{ds^2}V_2(x(s),y(s))\cos \theta(s)\theta'(s)
\right)\\
+ \varepsilon \left(\cos \theta(s)\theta'(s)^2\frac{d}{ds}V_1(x(s),y(s)) + \sin \theta(s)
\theta'(s)^2\frac{d}{ds}V_2(x(s),y(s)) \right)  + o(\varepsilon^2)\ ds
\end{multline*}
Thus
\begin{multline}\label{eq01}
E(\Omega_{\varepsilon})=\frac{1}{2} \int_0^{2\pi} \theta'(s)^2+\varepsilon\left( -3
\theta'(s)^2 \sin \theta(s) \frac{d}{ds}V_2(x(s),y(s))- 2\theta'(s)\sin \theta(s) \frac{d^2}{ds^2}V_1(x(s),y(s))\right.\\
\left.-3 \theta'(s)^2 \cos \theta(s)\frac{d}{ds}V_1(x(s),y(s)) + 2 \frac{d^2}{ds^2}V_2(x(s),y(s))\cos \theta(s)\theta'(s)
\right)+ o(\varepsilon^2)\ ds.
\end{multline}

Now,
\begin{equation}\label{eq02}
\int_0^{2\pi} \theta'(s)^2 \sin \theta(s) \frac{d}{ds}V_2(x(s),y(s)) \ ds= - \int_0^{2\pi} \left(2\theta' \theta'' \sin \theta + (\theta')^3 \cos\theta\right)V_2(x,y) \ ds,
\end{equation}
\begin{equation}\label{eq03}
\int_0^{2\pi} \theta'(s)^2 \cos \theta(s) \frac{d}{ds}V_1(x(s),y(s)) \ ds= - \int_0^{2\pi} \left(2\theta' \theta'' \cos \theta - (\theta')^3 \sin\theta\right)V_1(x,y) \ ds,
\end{equation}
\begin{equation}\label{eq04}
\int_0^{2\pi} \theta'(s)\sin \theta(s) \frac{d^2}{ds^2}V_1(x(s),y(s)) \ ds=  \int_0^{2\pi} \left( \theta'''\sin(\theta)+3\theta''\theta'\cos\theta -(\theta')^3\sin \theta \right)V_1(x,y) \ ds,
\end{equation}
\begin{equation}\label{eq05}
\int_0^{2\pi} \theta'(s)\cos \theta(s) \frac{d^2}{ds^2}V_2(x(s),y(s)) \ ds=  \int_0^{2\pi} \left( \theta'''\cos(\theta)-3\theta''\theta'\sin\theta -(\theta')^3\cos \theta \right)V_2(x,y) \ ds.
\end{equation}

Gathering \eqref{eq01} and \eqref{eq02}--\eqref{eq05} yields
\begin{multline}\label{eq06}
\frac{d E(\Omega_{\varepsilon})}{d \varepsilon} |_{\varepsilon =0} =
\frac 12\int_0^{2\pi}  \left(3 \left(2\theta' \theta'' \sin \theta + (\theta')^3 \cos\theta\right)V_2(x,y)
+3\left(2\theta' \theta'' \cos \theta - (\theta')^3 \sin\theta\right)V_1(x,y)
\right. \\ \left.
-2 \left( \theta'''\sin(\theta)+3\theta''\theta'\cos\theta -(\theta')^3\sin \theta \right)V_1(x,y)
+2\left( \theta'''\cos(\theta)-3\theta''\theta'\sin\theta -(\theta')^3\cos \theta \right)V_2(x,y)
\right) \ ds.
\end{multline}
Since the normal to $\Omega$ is $\nn=(\sin \theta, -\cos \theta)$, we deduce from \eqref{eq06} that
\begin{equation}\label{eq07}
\frac{d E(\Omega_{\varepsilon})}{d \varepsilon} |_{\varepsilon =0} =
-\int_0^{2\pi}  \left( \frac 12(\theta')^3 + \theta''' \right) \langle V, \nn\rangle \ ds.
\end{equation}

\section*{Acknowledgement}

This paper started while Chiara Bianchini was supported by the INRIA  research group {\it Contr\^ole robuste infini-dimensionnel et applications (CORIDA)} as post-doc. She is also supported by the Gruppo Nazionale per l'Analisi Matematica, la Probabilit\`a e le loro Applicazioni (GNAMPA) of the Istituto Nazionale di Alta Matematica (INdAM).

The work of Antoine Henrot and Tak\'eo Takahashi is supported by the project ANR-12-BS01-0007-01-OPTIFORM {\it Optimisation de formes} financed by the French Agence Nationale de la Recherche (ANR).

The work of Takeo Takahashi  is  also part of the INRIA project {\it Contr\^ole robuste infini-dimensionnel et applications (CORIDA)}.

The three authors had been supported by the  GNAMPA project  2013 {\it Metodi analitico geometrici per l'ottimizzazione di energie elastiche} in their visiting. 


\end{document}